\documentclass[11pt,reqno]{article}

\usepackage[usenames]{color}
\usepackage{graphicx}
\usepackage{amscd}
\usepackage[english]{babel}
\usepackage{color}
\usepackage{xcolor}
\usepackage{fullpage}
\usepackage{psfig}
\usepackage{graphics,amsmath,amssymb}
\usepackage{amsthm}
\usepackage{amsfonts}
\usepackage{latexsym}
\usepackage{epsf}
\usepackage{float}
\usepackage{MnSymbol}
\usepackage{makecell}
\usepackage{multirow}
\usepackage{longtable}
\usepackage{tabularx}

\usepackage[colorlinks=true,
linkcolor=webgreen,
filecolor=webbrown,
citecolor=webgreen]{hyperref}

\definecolor{webgreen}{rgb}{0,.5,0}
\definecolor{webbrown}{rgb}{.6,0,0}

\newcommand{\R}{{\mathbb R}}

\theoremstyle{plain}
\numberwithin{equation}{section}
\newtheorem{thm}{Theorem}[section]
\newtheorem{theorem}[thm]{Theorem}
\newtheorem{coro}[thm]{Corollary}

 \newcommand{\seqnum}[1]{\href{https://oeis.org/#1}{\underline{#1}}}

\newcommand{\eqn}[1]{(\ref{#1})}

\newcommand{\beql}[1]{\begin{equation}\label{#1}}
\newcommand{\eeq}{\end{equation}}

\begin{document}


\setcounter{page}{1}


\begin{center}

{\large\bf The Binary Two-Up Sequence  } \\
\vspace*{+.2in}

Michael De Vlieger, \\
Vinci Designs, 
5750 Delor St., St Louis, MO 63109, USA \\
Email:  \href{mailto:mike@vincico.com}{\tt mike@vincico.com}

\vspace*{+.1in}

Thomas Scheuerle, \\
Leonbergerstrasse 64/8, 71292 Friolzheim, GERMANY \\
Email:  \href{mailto:Thomas.Scheuerle@kabelbw.de}{\tt Thomas.Scheuerle@kabelbw.de}

\vspace*{+.1in}

R\'{e}my Sigrist, \\
3 rue de la Somme,
 67000 Strasbourg, FRANCE \\
 Email:  \href{mailto:remyetc9@gmail.com}{\tt remyetc9@gmail.com}
  
\vspace*{+.1in}

N. J. A. Sloane,\\ 
The OEIS Foundation Inc., 
11 South Adelaide Ave.,
Highland Park, NJ 08904, USA \\
Email:  \href{mailto:njasloane@gmail.com}{\tt njasloane@gmail.com}

\vspace*{+.1in}

Walter Trump,\\
 Reichelsdorfer Schulgasse 18, 90453 N\"{u}rnberg, GERMANY \\
Email:  \href{mailto:walter@trump.de}{\tt walter@trump.de}

\vspace*{+.3in}



\begin{abstract}
The Binary Two-Up Sequence is the lexicographically earliest sequence of distinct nonnegative integers with the property that the binary expansion of the $n$-th term has no $1$-bits in common with any of the previous $\lfloor \frac{n}{2} \rfloor$  terms.
We show that the sequence can be decomposed into ``atoms'', which are sequences of $4$, $6$, or $8$ numbers whose binary expansions match certain patterns, and that the sequence is the limiting form of a certain ``word'' involving the atoms. This leads to a fairly explicit formula for the terms, and in particular establishes the conjecture that every nonzero term is the sum of at most two powers of $2$.
\end{abstract}
\end{center}



\section{Overview}\label{Sec1}
We will say that two numbers $u$ and $v$ are {\em perpendicular}, written $u \perp v$,
if their binary representations have no $1$'s in the same bit-position.   For example, $9 = 1001_2 \perp 6 = 110_2$.
 The most explicit definition of our sequence $A = \{a(n), n \ge 0\}$ is:
 $a(0) = 0, a(1) = 1, a(2) = 2$, then,
 for each $k \ge  2$, given $a(k)$, the sequence is extended by appending two terms:
 $a(2k-1)$, which  is the  smallest integer $m \ge 0$ not among $a(0), \ldots, a(k)$ such that 
 $\{m, a(k), a(k+1), \ldots, a(2k-2)\}$ are pairwise perpendicular, and
 $a(2k)$, which is the smallest integer $m' \ge 0$ not among $a(0), \dots, a(k)$ such that 
$\{m', a(k), \ldots, a(2k-1)\}$ are pairwise perpendicular. 

Since it grows by two terms at each step, 
we call $A$ the {\em Binary Two-Up Sequence}.\footnote{It has nothing to do with the notorious Australian coin-tossing game of Two-Up.} It is entry \seqnum{A354169} in \cite{OEIS}.
 
In Section~\ref{Sec2} we discuss several equivalent definitions and study the first few terms.

In Section~\ref{Sec3} we examine the initial $2^{25}$ terms, and observe that they have a natural decomposition into five kinds of ``atoms'', labeled  $T, U, V, W, X$,  which are sequences of $4$, $6$, or $8$ numbers whose
binary expansions match certain patterns. 
The atoms are described in detail in Section~\ref{SecAtoms}.

Our main theorem (see Section~\ref{SecMainTh}) shows
that the sequence $A$ is the limiting form of a certain ``word'' involving the atoms. 
A corollary is that no term is the sum of more than two powers of $2$.

In Section~\ref{Sec5} we number all the atoms in $A$ in a systematic way.
Then  in Section~\ref{SecSpectra} we  use the main theorem to 
give the ``spectra'' of the atoms, that is, a precise list of the terms that make up each type of atom.
From this we can give a fairly explicit formula for the general term $a(n)$ of the sequence. 
It is not completely explicit, however, because in certain infrequent cases one of the terms in the atoms of type  $W$ involves a term in  an earlier atom of the same type. It is this modest amount of recursiveness that makes the sequence  interesting and yet solvable.
The analogous number-theoretic Two-Up Sequence \seqnum{A090252} has so far resisted 
all our attempts to analyze it.\footnote{Is there a way to define atoms there?  We do not know.}

In Section~\ref{SecSpectra} we also explain why the graph of the sequence has such a simple appearance.

In Section~\ref{SecSubsid},  we also use the  main theorem to establish  conjectured formulas for some related sequences, such as the indices of the terms that  are sums of two powers of $2$. 

\vspace{0.1in}
\noindent \textbf{Notation.} 
Throughout this article, the sequence \seqnum{A354169} will be denoted by $A$, and its $n$-th term by $a(n)$,  for $n \ge 0$.  Sequences with names like \seqnum{A354169}, that is, ``A'' followed by a six-digit number, refer to entries in  
the {\em On-Line Encyclopedia of Integer Sequences} \cite{OEIS}.

Two nonnegative integers $u$ and $v$ will be said to be ``perpendicular'', written $u \perp v$,  if their binary representations have no 1's in the same bit-position.   For example, $9 \perp 6$,  since their binary expansions $1001$ and $110$ have no common 1's. But 9 and 5 are not perpendicular, since $1001$ and $101$ do have a common 1.

We will often indicate the value of $a(n)$ by writing it as a sum of powers of 2 and listing the exponents in square brackets. For example $a(8) = 64 = 2^6 = [6]$, $a(9) = 12 = 2^2 + 2^3 = [2, 3]$.

The ``Hamming weight'' or simply  ``weight'' of a nonnegative number is the number of 1's in its binary expansion (cf. \seqnum{A000120}).

When we discuss the decomposition of $A$ into atoms, we will silently borrow notation from ``Combinatorics on 
Words'' \cite{Loth83}, and use ``words'' over the alphabet $\{T, U, V, W, X\}$. We  usually separate the letters (e.g. atoms)  in a word by either spaces or commas.

It  turns out that much of the structure of  $A$ is controlled by the sequence \seqnum{A029744}, which is a list of the numbers of the form $2^k$ and $3 \cdot 2^k$.  We denote the $m$-th term of this sequence by $\mu(m)$.
The initial values are:
\beql{Eqmu}
 \begin{array}{rrrrrrrrrrrrrrrrrrrrrrrr}
m: & 1 & 2 & 3 & 4 & 5 & 6 & 7 & 8 & 9 & 10 & 11 & 12 & 13 & 14 & 15 & 16 & 17 &  \cdots \\
\mu(m): & 1 & 2 & 3 & 4 & 6 & 8 & 12 & 16 & 24 & 32 & 48 & 64 & 96 & 128 & 192 & 256 & 384 & \cdots \\
\end{array}
\eeq
The formulas for the atoms (see Section~\ref{SecSpectra}) make heavy use of  this function.


\section{Definition of the Binary Two-Up Sequence}\label{Sec2}

There are several equivalent definitions of the sequence. (The equivalence of the four definitions is easily established.) The first definition, the ``Two-Up'' version,  is the most explicit.

\textbf{Definition 1:}
$a(0) = 0, a(1) = 1, a(2) = 2$; then,
        for each $k \ge  2$, given $a(k)$, the sequence is extended by appending two terms:
        $a(2k-1)$, which  is the  smallest integer $m \ge 0$ not among $a(0), \ldots, a(k)$ such that 
        $\{m, a(k), a(k+1), \ldots, a(2k-2)\}$ are pairwise perpendicular, and
        $a(2k)$, which is the smallest integer $m' \ge 0$ not among $a(0), \dots, a(k)$ such that $\{m', a(k), \ldots, a(2k-1)\}$ are pairwise perpendicular.
        
We can think of this process as an algorithm which takes the terms  $a(0), a(1), \ldots, a(k)$  as input and produces  $a(0), a(1), \ldots, a(2k-1), a(2k)$  as output. A354169 is the limiting sequence as $k$ goes to infinity.

A consequence of Definition 1, and also an equivalent definition, is:

\textbf{Definition 2:}
         $A$ is the lexicographically earliest infinite sequence of distinct nonnegative integers with the property that $a(n)$ is perpendicular to each of the next $n$ terms.

(See \cite{Yellow} for background information about ``lexicographically earliest sequences''.)

Another equivalent definition is:

\textbf{Definition 3:}
        $A$ is the lexicographically earliest infinite sequence of distinct nonnegative integers with the property that $a(n)$ is perpendicular to each of the previous $\lfloor n/2 \rfloor$ terms.

Yet another equivalent version, perhaps the easiest to remember, is:

\textbf{Definition 4:}
\begin{align}\label{Eq1}
\mbox{For~} n \ge 0, a(n) & \mbox{~is~the~smallest~nonnegative~integer~not~yet~in~the~sequence}  \nonumber  \\
& \mbox{which~is~perpendicular~to~each~of~} a(\lceil n/2 \rceil),   \ldots, a(n-1). 
\end{align}

Terms $a(0)$ through $a(17)$ of $A$ are:
$$
0, 1, 2, 4, 8, 3, 16, 32, 64, 12, 128, 256, 512, 17, 1024, 34, 2048, 4096.
$$
One can see that these terms are a mixture of 0, powers of 2, and sums of two smaller powers of 2. It will turn out that this holds in general, although we will not know that until we have proved the main theorem.

The initial terms of $A$ are somewhat exceptional. Although the decomposition into atoms begins at $a(1)$, 
there are irregularities which persist until $a(565)$, and
our explicit formulas for the terms in  the atoms (see Section~\ref{SecSpectra}) assume $n \ge 566$.  The OEIS entry \seqnum{A354169} contains an explicit list of the first 4941 terms, and further information can be found in  
sequences
\seqnum{A354680}, \seqnum{A354767}, \seqnum{A354773}, \seqnum{A354774}, \seqnum{A354798}, \seqnum{A355150}.

Table~\ref{Tab18terms} shows  
shows the construction of the initial terms, working directly from Definition~1.
We are given $a(0)=0$, $a(1)=1$, and $a(2)=2$. Taking $a(2)$ as the controlling term,
first $a(3)$ must be perpendicular to $a(2)$, and we get $a(3)=4$, and then $a(4)$ must be perpendicular to 
$a(2)$ and $a(3)$, and we get $a(4)=8$.
Next we take $a(3)=4$ as the controlling term, and get $a(5)$ and $a(6)$, and so on.
The columns in Table~\ref{Tab18terms}, reading from right to left,  give $n$, $a(n)$, the term $a(k)$ that is controlling the terms $a(2k-2)$ and $a(2k-1)$, and $a(n)$ written in base~$2$.\footnote{A  larger illustration showing the binary expansion of the first $190$ terms can be found here:
\url{https://oeis.org/A354169/a354169_1.pdf}.}

\begin{table} [!htb]
\caption{The start of the construction of the sequence $A$.
Reading from right to left,  the columns give $n$, $a(n)$, the term $a(k)$ that is controlling the terms $a(2k-2)$ and $a(2k-1)$, and $a(n)$ written in base $2$.}\label{Tab18terms}
$$
\begin{array}{|rrrrrrrrrrrrrr|r|r|r|}
\hline
\multicolumn{14}{|c|}{a(n) \mbox{~in~binary}} & a(k) & a(n) & n \\
\hline
~  & ~  & ~  & ~  & ~  & ~  & ~  & ~  & ~  & ~  & ~  & ~  & ~  & 0  &   ~ & 0 &  0 \\
~  & ~  & ~  & ~  & ~  & ~  & ~  & ~  & ~  & ~  & ~  & ~  & ~  & 1  &   ~ & 1 &  1 \\
~  & ~  & ~  & ~  & ~  & ~  & ~  & ~  & ~  & ~  & ~  & ~  & 1  & 0  &   ~ & 2 &  2 \\
\hline
~  & ~  & ~  & ~  & ~  & ~  & ~  & ~  & ~  & ~  & ~  & 1  & 0  & 0  &   2 & 4 &  3 \\
~  & ~  & ~  & ~  & ~  & ~  & ~  & ~  & ~  & ~  & 1  & 0  & 0  & 0  &   ~ & 8 &  4 \\
\hline
~  & ~  & ~  & ~  & ~  & ~  & ~  & ~  & ~  & ~  & ~  & ~  & 1  & 1  &   4 & 3 &  5 \\
~  & ~  & ~  & ~  & ~  & ~  & ~  & ~  & ~  & 1  & 0  & 0  & 0  & 0  &   ~ & 16 &  6 \\
\hline
~  & ~  & ~  & ~  & ~  & ~  & ~  & ~  & 1  & 0  & 0  & 0  & 0  & 0  &   8 & 32 &  7 \\
~  & ~  & ~  & ~  & ~  & ~  & ~  & 1  & 0  & 0  & 0  & 0  & 0  & 0  &   ~ & 64 &  8 \\
\hline
~  & ~  & ~  & ~  & ~  & ~  & ~  & ~  & ~  & ~  & 1  & 1  & 0  & 0  &   3 & 12 &  9 \\
~  & ~  & ~  & ~  & ~  & ~  & 1  & 0  & 0  & 0  & 0  & 0  & 0  & 0  &   ~ & 128 &  10 \\
\hline
~  & ~  & ~  & ~  & ~  & 1  & 0  & 0  & 0  & 0  & 0  & 0  & 0  & 0  &   16 & 256 &  11 \\
~  & ~  & ~  & ~  & 1  & 0  & 0  & 0  & 0  & 0  & 0  & 0  & 0  & 0  &   ~ & 512 &  12 \\
\hline
~  & ~  & ~  & ~  & ~  & ~  & ~  & ~  & ~  & 1  & 0  & 0  & 0  & 1  &   32 & 17 &  13 \\
~  & ~  & ~  & 1  & 0  & 0  & 0  & 0  & 0  & 0  & 0  & 0  & 0  & 0  &   ~ & 1024 &  14 \\
\hline
~  & ~  & ~  & ~  & ~  & ~  & ~  & ~  & 1  & 0  & 0  & 0  & 1  & 0  &   64 & 34 &  15 \\
~  & ~  & 1  & 0  & 0  & 0  & 0  & 0  & 0  & 0  & 0  & 0  & 0  & 0  &   ~ & 2048 &  16 \\
\hline
~  & 1  & 0  & 0  & 0  & 0  & 0  & 0  & 0  & 0  & 0  & 0  & 0  & 0  &   12 & 4096 &  17 \\
1  & 0  & 0  & 0  & 0  & 0  & 0  & 0  & 0  & 0  & 0  & 0  & 0  & 0  &   ~ & 8192 &  18  \\
\hline
\end{array}
$$
\end{table}

\textbf{Remark 2.1:} The perpendicular condition states that:\\
  \indent  ~~~$a(2k-1) \perp ~a(k), \ldots, a(2k-2)$, and\\
  \indent  ~~~$a(2k) ~~~~\perp ~a(k), \ldots, a(2k-1)$. 
  
In particular, writing $k$ for $a(k)$ (so that the third line, for example, means $a(4) \perp a(2), a(3)$), we have:
$$
\begin{array} {crrrrrrrr}
2 & \perp & 1 \\
3 & \perp & 2 \\
4 & \perp & 2 & 3 \\
5 & \perp & 3 & 4 \\
6 & \perp & 3 & 4 & 5 \\
7 & \perp & 4 & 5 & 6 \\
8 & \perp & 4 & 5 & 6 & 7 \\
9 & \perp & 5 & 6 & 7 & 8 \\
10 & \perp & 5 & 6 & 7 & 8 & 9 \\
11 & \perp & 6 & 7 & 8 & 9 & 10 \\
12 & \perp & 6 & 7 & 8 & 9 & 10 & 11 \\
13 & \perp & 7 & 8 & 9 & 10 & 11 & 12 \\
14 & \perp & 7 & 8 & 9 & 10 & 11 & 12 & 13 \\
15 & \perp & 8 & 9 & 10 & 11 & 12 & 13 & 14 \\
...
\end{array}$$
\noindent
(It is  helpful to have these rules stated explicitly when  studying the sequence.)

\textbf{Remark 2.2:}  \seqnum{A354169} is a set-theory analog of \seqnum{A090252}. It bears the same relation to \seqnum{A090252} as \seqnum{A252867} does to \seqnum{A098550}, \seqnum{A353708} to \seqnum{A121216}, \seqnum{A353712} to \seqnum{A347113},
etc.


\section{Properties of  A, the construction log, and  ``atoms''}\label{Sec3}


\begin{table}[!ht]
\caption{Log file for $A354169$.}\label{TableLog}
\scriptsize
$$
 \begin{array} {llll}
 \hline
A1 & \mbox{ found } & a(1)=[0] & \mbox{start of atom } T \\
A1 & \mbox{ found } & a(2)=[1] \\
F1 & \mbox{ freed } & a(1)=[0] \\
A1 & \mbox{ found } & a(3)=[2] \\
A1 & \mbox{ found } & a(4)=[3] \\
F1 & \mbox{ freed } & a(2)=[1]  & \mbox{end of atom } T \\
 \hline
A2 & \mbox{ found } & a(5)=[0, 1] & \mbox{start of atom } U(0) \\
A1 & \mbox{ found } & a(6)=[4] & \\
F1 & \mbox{ freed } & a(3)=[2] \\
A1 & \mbox{ found } & a(7)=[5] \\
A1 & \mbox{ found } & a(8)=[6] \\
F1 & \mbox{ freed } & a(4)=[3] &  \mbox{end of atom } U(0) \\
 \hline
A2 & \mbox{ found } & a(9)=[2, 3] & \mbox{start of atom } V(-1) \\
A1 & \mbox{ found } & a(10)=[7] \\
F2 & \mbox{ freed } & a(5)=[0, 1] \\
A1 & \mbox{ found } & a(11)=[8] \\
A1 & \mbox{ found } & a(12)=[9] \\
F1 & \mbox{ freed } & a(6)=[4] \\
A2 & \mbox{ found } & a(13)=[0, 4] \\
A1 & \mbox{ found } & a(14)=[10] \\
F1 & \mbox{ freed } & a(7)=[5]  & \mbox{end of atom } V(-1) \\
 \hline
A2 & \mbox{ found } & a(15)=[1, 5] & \mbox{start of atom } W(-1) \\
A1 & \mbox{ found } & a(16)=[11] \\
F1 & \mbox{ freed } & a(8)=[6] \\
A1 & \mbox{ found } & a(17)=[12] \\
A1 & \mbox{ found } & a(18)=[13] \\
F2 & \mbox{ freed } & a(9)=[2, 3] \\
A2 & \mbox{ found } & a(19)=[2, 6] \\
A1 & \mbox{ found } & a(20)=[14] \\
F1 & \mbox{ freed } & a(10)=[7]  & \mbox{end of atom } W(-1) \\
 \hline
A2 & \mbox{ found } & a(21)=[3, 7] & \mbox{start of atom } U(1) \\
A1 & \mbox{ found } & a(22)=[15] \\
F1 & \mbox{ freed } & a(11)=[8] \\
A1 & \mbox{ found } & a(23)=[16] \\
A1 & \mbox{ found } & a(24)=[17] \\
F1 & \mbox{ freed } & a(12)=[9] & \mbox{end of atom } U(1) \\
 \hline
A2 & \mbox{ found } & a(25)=[8, 9] & \mbox{start of atom } X(2) \\
A1 & \mbox{ found } & a(26)=[18]  & \\
F2 & \mbox{ freed } & a(13)=[0, 4] \\
A1 & \mbox{ found } & a(27)=[19] \\
A1 & \mbox{ found } & a(28)=[20] \\
F1 & \mbox{ freed } & a(14)=[10] \\
A2 & \mbox{ found } & a(29)=[0, 10] \\
A1 & \mbox{ found } & a(30)=[21] \\
F2 & \mbox{ freed } & a(15)=[1, 5] \\
A2 & \mbox{ found } & a(31)=[1, 4] \\
A1 & \mbox{ found } & a(32)=[22] \\
F1 & \mbox{ freed } & a(16)=[11] & \mbox{end of atom } X(2) \\
 \hline
A2 & \mbox{ found } & a(33)=[5, 11] & \mbox{start of atom } U(2) \\
A1 & \mbox{ found } & a(34)=[23] & \\
F1 & \mbox{ freed } & a(17)=[12] \\
\ldots &\ldots & \ldots & \ldots
 \end{array}
$$
\normalsize
\end{table}

Starting in this section, we will usually ignore the initial term $a(0)=0$, since it plays no role in the main theorem and,
since it is not a sum of powers of $2$,  would just complicate the proof. So from now on we will usually think of $A$ as consisting of the terms $\{a(n), n \ge 1\}$.\footnote{The term $a(0)=0$ exists because the sequence is a set-theory based lexicographically  earliest sequence, and the initial term represents the empty set.}

We start with some general remarks about the sequence. 

It is easy to see that $A$ exists. This is follows from the fact that there is always a candidate for $a(n)$, namely that power of $2$ that is just greater than all the previous terms. So by Definition 4, $a(n)$ always exists.

Furthermore, since now we know the sequence is infinite and contains no repeated terms, the binary expansions of the terms must contain higher and higher powers of $2$. The first time a term contains $2^e$ in its binary expansion, that term will actually equal $2^e$ (since we always choose the smallest possible value).  We can conclude that for every $e \ge 0$  there is an $n$ such that $a(n) = 2^e$ (see \seqnum{A354767}), and that these powers of 2 appear in increasing order.

\subsection{The log file.}\label{SecLog}
In order for a term $a(n)$ that is not a power of $2$ to appear, the 1's in its binary expansion have to avoid the 1's in the binary expansions of $a(\lceil\frac{n}{2}\rceil)$  through $a(n-1)$. Understanding how this happens is the key to understanding the sequence, and in order to do so it is helpful to keep a log of the construction process. The beginning of this log is shown  in Table~\ref{TableLog}.

The log records when a new value of $a(n)$ is found, and when the powers of $2$ in an earlier term are free to be reused. From~\eqn{Eq1} we see that once the value of $a(2t)$ has been determined, $a(2t+1), a(2t+2), \ldots$ no longer have to be perpendicular to $a(t)$, and so the powers of $2$ in the binary expansion of $a(t)$ can be freed for reuse.
We keep track of this by recording which $1$'s in the binary expansions can be reused.

The log has two kinds of entries:  $Aw$ means that a term of weight $w$ has been found, and   $Fw$ means the exponents in an earlier term of weight $w$ have been freed.
 
So a line like\footnote{In the log and in most of the later tables, we represent numbers using the square bracket notation mentioned in the Introduction.}
$$
   A1 \mbox{~found~} a(8) = [6]
$$
in the log indicates that we have found $a(8) = 2^6$, $A1$ indicating that this has weight 1.
A line like
$$
   A2 \mbox{~found~} a(9) = [2,  3]
$$
indicates that we have found $a(9) = 2^2 + 2^3 = 12$, $A2$ indicating that this has weight 2.
If terms of weight greater than 2 appeared, they would be indicated by $A3, A4, \ldots$. They are permitted by the notation, but as we will show, they do not happen.

A line like
$$
   F1 \mbox{~freed~} a(4) = [3]
$$
indicates the term $2^3$ in $a(4)$ is now free to be reused.
A line like
$$
   F2 \mbox{~freed~} a(5) = [0, 1]
$$
indicates the powers of $2$ $2^0$ and $2^1$ in $a(5)$ are now free to be reused.
If terms of weight greater than 2 were freed, they would be indicated by $F3, F4, \ldots$. They are permitted by the notation, but do not happen.

\subsection{The construction.}\label{SecConstruction}
The construction of $A$ can now be restated as follows:
\begin{itemize}
\item We keep track of the set of free 1's.
\item We check if a subset of the free 1's can be added to form a value not yet in the sequence, and if so  take the smallest such value to be the next term (this value necessarily has weight at least 2).
\item If that is not the case, the next term is a new power of 2.
\end{itemize}


\begin{table}[!ht]
  \caption{Construction of first six terms of $A$.}\label{Tab6terms}
$$
\begin{array}{|c|c|l|}
\hline
n & a(n) & \mbox{ free } 1's \\
\hline
~ & ~ &  \{ \}  \\
\hline
1 & \color{red}{2^0}  & ~ \\
\hline
2 & {\color{blue}{2^1}}  & ~ \\
\hline
~ & ~ & \{ {\color{red}{2^0}} \} ( {\color{red}{2^0}} \mbox{ comes from~} a(1)) \\
\hline
3 & 2^2 & ~ \\
\hline
4 & 2^3 & ~ \\
\hline
~ & ~ & \{{\color{red}{2^0}} , {\color{blue}{2^1}} \} (2^1 \mbox{ comes from~} a(2))  \mbox{ [See Note]} \\
\hline
\hline
5 &  {\color{red}{2^0}}  +{\color{blue}{2^1}}  & ~ \\
\hline
~ & ~ & \{ \}  \mbox{ No free $1$'s remain.} \\
\hline
6 & 2^4 & ~ \\
\hline
\end{array}
$$
[Note: At this point, we have two free $1$'s. As they came from terms that are powers of $2$, they have never been combined before, and we can add them to produce the next term. The first atom, $T$, ends here.]
\end{table}

\noindent
\textbf{Remarks about the freeing and construction steps.}
\begin{enumerate}
\item Table~\ref{Tab6terms}  illustrates the process for the first six terms of $A$.
\item The new powers of 2 occur in order, and no power of 2 is skipped.
\item In order for a term of weight 2 or more to exist, there must be at least two free powers of 2.
\item The freed terms are just the terms of the sequence itself, at about $n/2$ steps earlier, taken in order, and used exactly once.
\item As long as we have not yet seen any term of weight greater than 2, the freed terms will also have weight at most 2.  As long as we have not seen two weight 2 terms in succession, after a weight 2 term has been freed, the next freed term will have weight 1.
\item When a weight 1 term $2^e$, say, is freed , we can be certain that up to this point the sequence does not contain a term $2^e + x, x \ne 0, x \ne 2^e$. This observation will often be useful for deciding which candidate for $a(n)$ is the smallest.
\item In the first $2^{25}$ terms there are never more than three free powers of 2 at any time (it will turn out that this is always true). If there are three free powers of 2 there is always a unique way to add two of them to get the next term.
\end{enumerate}

\subsection{Atoms.}\label{SecAtoms}

We split the sequence $A$ into blocks called ``atoms'', by declaring that an atom ends whenever there are exactly two distinct free 1's that can be added to form a new term. That is, the current atom ends just before those two 1's are combined. We indicate the end of an atom by drawing a horizontal line in the log.


\begin{table}
\caption{The codes for the five types of atoms.}\label{TableCodes}
\begin{align}\label{EqAtoms}
T & ~= A1 A1 F1 A1 A1 F1  \,,  \nonumber \\
U & ~= A2 A1 F1 A1 A1 F1 \,,  \nonumber \\
V & ~= A2 A1 F2 A1 A1 F1 A2 A1 F1 \,,  \nonumber \\
W & ~= A2 A1 F1 A1 A1 F2 A2 A1 F1  \,,  \nonumber \\
X & ~= A2 A1 F2 A1 A1 F1 A2 A1 F2 A2 A1 F1 \,.
\end{align}
\end{table}

We describe the atoms by listing the sequence of symbols at the start of the lines in the log.
We will refer to this sequence as the ``code'' for the atom.
The first atom, for example, as we can see from Table~\ref{TableLog}, has  code $A1, A1, F1, A1, A1, F1$.
We call it an atom of type $T$. It is special, and only occurs once, at the start of the sequence.
(This is the only time there are four powers of $2$ in succession.)
The second atom is $A2, A1, F1, A1, A1, F1$, which we call a U atom.

When we examine the log file (see Table~\ref{TableLog}), we observe that just five types of atoms occur, which we denote by $T$, $U$, $V$, $W$, and $X$.

Here are the formal definitions of the atoms $T$, $U$, $V$, $W$, and $X$.  
They are subsequences of~$A$, always beginning with a term $a(i)$ where $i$ is odd, and characterized by 
a code which lists the pattern of weights and freed numbers (see Table~\ref{TableCodes}). 
In the code, the symbol $Aw$ means that a term of weight $w$ is created, 
and $Fw$ means that a term of weight $w$ is freed. Except for the $T$ atom, 
each atom begins with two free $1$'s, which are used to create $a(i)$. 
Now there are no free $1$'s, so $a(i+1)$ is the next free power of $2$. 
So $U$, $V$, $W$, and $X$ atoms all begin with $A2$, $A1$. 
Each atom ends with exactly two free $1$'s.

In Table~\ref{TableWts} we record  the (necessarily even) lengths of the atoms and  the successive Hamming weights of their terms. They have even lengths because numbers are freed only when $n$ is even.


\begin{table}[!ht]
  \caption{Lengths and weights of the five types of atoms.}\label{TableWts}
$$
\begin{array}{|c|c|l|}
\hline
\mbox{Atom} & \mbox{Length} & \mbox{Weights, in order of occurrence } \\
\hline
T & 4  & 1, 1, 1, 1 \\
\hline
U & 4 & 2, 1, 1, 1 \\
\hline
V & 6 & 2, 1, 1, 1, 2, 1 \\
\hline
W & 6 & 2, 1, 1, 1, 2, 1 \\
\hline
X & 8 &  2, 1, 1, 1, 2, 1, 2, 1 \\
\hline
\end{array}
$$
\end{table}

Note that the list of weights alone does not specify the atoms uniquely.
For example the weight sequence $2, 1, 1, 1, 2, 1$  could correspond to a $U$ atom (with extra terms), a $V$ atom, a $W$ atom, or the beginning of an $X$ atom. To determine the correct alternative we must consider the free 1's that appear in the atoms. We do this in the next five tables.

\vspace{0.1in}
\noindent \textbf{The $T$ atom: see Table~\ref{TableT}.} 



\begin{table}[!ht]
  \caption{The unique $T$ atom, terms $a(1)-a(4)$.}\label{TableT}
$$
\begin{array}{|c|l|c|}
\hline
\mbox{Input terms (supplying free 1's)} & \mbox{Atom }  T & \mbox{Weights } \\
\hline
~ & \mbox{At this point we have no free 1's} & ~ \\
\hline
a(1) = {\color{red}{2^0}} & a(1) = 2^0 &  1 \\ \cline{2-3} 
                  & a(2) = 2^1 & 1 \\
\hline
a(2) = {\color{blue}{2^1}} & a(3) = 2^2 &  1 \\ \cline{2-3} 
                  & a(4) = 2^3 & 1 \\
\hline
~ & \mbox{There are now two free 1's}, {\color{red}{2^0}} \mbox{ and } {\color{blue}{2^1}}  & ~ \\
\hline
\end{array}
$$
\end{table}

\vspace{0.1in}
\noindent \textbf{The $U$ atoms.} 

 An example is shown in Table~\ref{TableU}. (This is $U(2)$ in the numbering system introduced in Section~\ref{Sec5}).


\begin{table}[!ht]
  \caption{An atom of type $U$, terms $a(21)-a(24)$.}\label{TableU}
$$
\begin{array}{|c|l|c|}
\hline
\mbox{Input terms (supplying free 1's)} & \mbox{Atom }  $U$ & \mbox{Weights } \\
\hline
~ & \mbox{At this point we have two free 1's }, 2^3 \mbox{ and } 2^7 & ~ \\
\hline
a(11) = {\color{red}{2^8}} & a(21) = 2^3 + 2^7 &  2 \\ \cline{2-3} 
                  & a(22) = 2^{15} & 1 \\
\hline
a(12) = {\color{blue}{2^9}} & a(23) = 2^{16} &  1 \\ \cline{2-3} 
                  & a(24) = 2^{17} & 1 \\
\hline
~ & \mbox{There are now two free 1's}, {\color{red}{2^8}} \mbox{ and } {\color{blue}{2^9}}  & ~ \\
\hline
\end{array}
$$
\end{table}

\vspace{0.1in}
\noindent \textbf{The $V$ atoms.} 

 An example (this is $V(-1)$) is shown in Table~\ref{Table5V}.
 

\begin{table}[!ht]
  \caption{An atom of type $V$, terms $a(9)-a(14)$.}\label{Table5V}
$$
\begin{array}{|c|l|c|}
\hline
\mbox{Input terms (supplying free 1's)} & \mbox{Atom }  $V$ & \mbox{Weights } \\
\hline
~ & \mbox{At this point we have two free 1's }, 2^2 \mbox{ and } 2^3 & ~ \\
\hline
a(5) = {\color{red}{2^0}} + {\color{blue}{2^1}}& a(9) = 2^2 + 2^3 &  2 \\ \cline{2-3} 
                  & a(10) = 2^{7} & 1 \\
\hline
a(6) = {\color{cyan}{2^4}}  & a(11) = 2^{8} &  1 \\ \cline{2-3} 
                  & a(12) = 2^{9} & 1 \\
\hline
a(7) = {\color{purple}{2^5}} & a(13) = {\color{red}{2^0}}   + {\color{cyan}{2^4}}&  2 \\ \cline{2-3} 
                  & a(14) = 2^{10} & 1 \\
\hline
~ & \mbox{There are now two free 1's}, {\color{blue}{2^1}} \mbox{ and } {\color{purple}{2^5}}  & ~ \\
\hline
\end{array}
$$
\end{table}

\vspace{0.1in}
\noindent \textbf{The $W$ atoms.} 

 An example  (this is $W(-1)$) is shown in Table~\ref{TableW}.


\begin{table}[!ht]
  \caption{An atom of type $W$, terms $a(15)-a(20)$.}\label{TableW}
$$
\begin{array}{|c|l|c|}
\hline
\mbox{Input terms (supplying free 1's)} & \mbox{Atom }  $W$ & \mbox{Weights } \\
\hline
~ & \mbox{At this point we have two free 1's }, 2^1 \mbox{ and } 2^5 & ~ \\
\hline
a(8) = {\color{red}{2^6}} & a(15) = 2^1 + 2^5 &  2 \\ \cline{2-3} 
                  & a(16) = 2^{11} & 1 \\
\hline
a(9) = {\color{blue}{2^2}} + {\color{cyan}{2^3}}  & a(17) = 2^{12} &  1 \\ \cline{2-3} 
                  & a(18) = 2^{13} & 1 \\
\hline
a(10) = {\color{purple}{2^7}} & a(19) = {\color{blue}{2^2}}   + {\color{red}{2^6}}&  2 \\ \cline{2-3} 
                  & a(20) = 2^{14} & 1 \\
\hline
~ & \mbox{There are now two free 1's}, {\color{cyan}{2^3}} \mbox{ and } {\color{purple}{2^7}}  & ~ \\
\hline
\end{array}
$$
\end{table}

\vspace{0.1in}
\noindent \textbf{The $X$ atoms.} 

 An example (this is $X(2)$)  is shown in Table~\ref{TableX}.


\begin{table}[!ht]
  \caption{An atom of type $X$, terms $a(25)-a(32)$.}\label{TableX}
$$
\begin{array}{|c|l|c|}
\hline
\mbox{Input terms (supplying free 1's)} & \mbox{Atom }  $X$ & \mbox{Weights } \\
\hline
~ & \mbox{At this point we have two free 1's }, 2^8 \mbox{ and } 2^9 & ~ \\
\hline
a(13) = {\color{red}{2^0}} + {\color{blue}{2^4}} & a(25) = 2^8 + 2^9 &  2 \\ \cline{2-3} 
                  & a(26) = 2^{18} & 1 \\
\hline
a(14) = {\color{olive}{2^{10}}}   & a(27) = 2^{19} &  1 \\ \cline{2-3} 
                  & a(28) = 2^{20} & 1 \\
\hline
a(15) = {\color{purple}{2^1}} + {\color{cyan}{2^5}} & a(29) = {\color{red}{2^0}} + {\color{olive}{2^{10}}} &  2 \\ \cline{2-3} 
                  & a(30) = 2^{21} & 1 \\
\hline
a(16) = {\color{teal}{2^{11}}} & a(31) = {\color{purple}{2^1}}   + {\color{blue}{2^4}}&  2 \\ \cline{2-3} 
                  & a(32) = 2^{22} & 1 \\
\hline
~ & \mbox{There are now two free 1's}, {\color{cyan}{2^5}} \mbox{ and } {\color{teal}{2^{11}}}  & ~ \\
\hline
\end{array}
$$
\end{table}


\newpage
\section{The main theorem}\label{SecMainTh}

Our main theorem specifies $A$ as a sequence of atoms.

\begin{theorem}\label{MainTh}
The sequence $A$ is the limit as $k$ goes to infinity of $S R(1) R(2) \ldots R(k)$, where $S = T U V W U$ and
\beql{EqRk}
R(k) ~=~ (V W)^{2^{k-1}-1} X\, U \,(V W)^{2^{k-1}-1} \,X\, U
\eeq 
for $k \ge 1$.
Furthermore, in $R(2)$ and beyond, 
any pair of consecutive atoms $X\,U$ has the following ``ancestor property'': 
the  last six terms in such a group, say:
$$
\left. 
\begin{array}{ll}
a(p) &= 2^t + 2^u, \\
a(p+1) &= 2^y,
\end{array}
 \right\} \mbox{~from~atom } X
 $$
 and
 $$
 \left.
\begin{array}{ll}
a(p+2) &= 2^v + 2^w, \\
a(p+3) &= 2^{y+1}, \\
a(p+4) &= 2^{y+2}, \\
a(p+5) &= 2^{y+3}, \\
\end{array}
 \right\} \mbox{~from~atom } U
$$
where $t < u, v < w$ and $u < w$,
have a common ``ancestor'' $a(m) = 2^u + 2^v$,  and $2^u$ and $2^w$ can be added to form a new term.
\end{theorem}

\subsection{The local algorithm.}\label{SecLocal}
For the proof we will make use of a ``local'' version of the algorithm that generates $A$.
That algorithm,  described following  Definition 1 at the beginning of this article
and again at the start of Section~\ref{SecConstruction},
takes as input an initial segment $a(1), \ldots, a(n)$ of $A$, and produces as output 
 the  segment $a(1), \ldots, a(2n)$ of twice the length.

The local algorithm operates on finite segments of $A$, 
 but these segments do not need to start at the beginning of $A$. 
 However, the input and output must be sequences of atoms.
 By definition, an atom ends with two free numbers, whose sum is the first term of the atom.
 Suppose the input is a sequence of atoms $I = I(1), I(2), \ldots, I(k)$, and the output is
 a sequence of atoms $O = O(1), O(2), \ldots,  O(\ell)$. The algorithm knows $I(1)$, 
 but (assuming we are not at the very beginning of $A$), all it knows about $O(1)$ is that the first two terms are $2^t + 2^u$ and $2^z$
 for some integers $t, u, z$. 
 

\begin{table}[!ht]
\caption{Illustrating the local algorithm:   inputs $V$ or $W$ produce output $V W$. }\label{MyFig4b}
$$
\begin{array}{|c|c|c|c|c|c|c|c|}\hline
\multicolumn{3}{|c}{\mbox{Input}} &  \multicolumn{5}{|c|}{\mbox{Output}} \\ \hline
\mbox{atom} & \mbox{index} & I(n) & \mbox{index} & O(n) & \mbox{free} & \mbox{code} & \mbox{atom} \\
\hline
V \mbox{ or } W & n    & [b,c]  & 2n-1 & [t,u]  & -          & A2       & V \\ \cline{4-7}
   & ~    & b < c & 2n    & [z]     & [b,c]    & A1, F2 & ~ \\  \cline{2-7} 
   & n+1 & [d]   & 2n+1 & [z+1]  & [b,c]   & A1 &    \\ \cline{4-7}
   & ~     &        & 2n+2  & [z+2]  & [b,c,d] & A1, F1 & ~ \\  \cline{2-7} 
   & n+2 & [d+1]  & 2n+3 & [b,d] & [c]        & A2 &    \\  \cline{4-7}
   & ~ &              & 2n+4  & [z+3]  & [c,d+1] & A1, F1 & ~ \\  \cline{2-8} 
   & n+3 & [d+2]  & 2n+5 & [c,d+1] & -       & A2 &  W  \\ \cline{4-7}
   & ~ &              & 2n+6  & [z+4]  & [d+2] & A1, F1 & ~ \\  \cline{2-7}
   & n+4 & [e,f]  & 2n+7 & [z+5] & [d+2]      & A1 &    \\ \cline{4-7}
   & ~ &    e<f    & 2n+8  & [z+6]  & [d+2, e,f] & A1,F2 & ~ \\  \cline{2-7}
   & n+5 & [d+3]  & 2n+9 & [d+2,e] & [f]       & A2 &    \\ \cline{4-7}
   & ~ &             & 2n+10  & [z+7]  & [f,d+3] & A1,F1 & ~ \\  
\hline
\end{array}
$$
\end{table}

From then on the local algorithm proceeds in the same way as
 the main algorithm. At step $2n$ in the output, term $n$ in the input is freed for reuse,
 and the next output is either the next power of $2$ or the smallest sum of a subset of the available free terms that has not yet been used.  Except when there is a pair of input atoms  $X, U$, Remark 6 at the end of Section~\ref{SecConstruction}
 always guarantees that any weight $2$ term we want to use in the output has not yet
 appeared in the sequence. In the case of input atoms $X, U$, the ``ancestor property'' provides the guarantee that we need (see Tables~\ref{TableSR1} and \ref{TableXUgives}  in the Appendix).
 
We require that when the input atoms have all been read, the output must be a sequence of  atoms, that is, a sequence of complete atoms, fragments of atoms are not permitted. The output will contain twice as many terms of $A$ as the input. We write $I \rightarrow O$ to indicate the action of the algorithm.
 With our restrictions on inputs and outputs, it follows that if $I_1 \rightarrow O_1$
 and $I_2 \rightarrow O_2$, then $I_1 I_2 \rightarrow O_1 O_2$.
 
 Table~\ref{MyFig4b} illustrates the local algorithm by showing how an input consisting
 of a single atom $V$ or $W$ produces the output $V W$. We keep track of the code for each step, because that is 
 how we distinguish a $W$ atom from a $V$ atom.
 
 Similarly, the single input atom $T$ produces the output $T U$ (see Table~\ref{TabTtoTU}).

 
 \begin{table}[!ht]
\caption{Illustrating the local algorithm:   input $T$ produces output $T U$. }\label{TabTtoTU}
$$
\begin{array}{|c|c|c|c|c|c|c|}\hline
\multicolumn{3}{|c}{\mbox{Input}} &  \multicolumn{4}{|c|}{\mbox{Output}} \\ \hline
\mbox{atom} & \mbox{index} & I(n) & \mbox{index} & O(n) & \mbox{free} & \mbox{atom} \\
\hline
T & 1    & 2^0  & 1 & 2^0             & -       & T \\ \cline{4-6}
   & ~    &        & 2  & 2^1             & 2^0  & ~ \\  \cline{2-6}   
   & 2   & 2^1   & 3 & 2^2             & 2^0   &    \\ \cline{4-6}
   & ~   &         & 4  & 2^3        & 2^0, 2^1 & ~ \\  \cline{2-7} 
   & 3   & 2^2   & 5  & 2^0+2^1 = 3 & - & U \\  \cline{4-6}
   & ~   &         & 6  & 2^4               & 2^2   & ~ \\  \cline{2-6} 
   & 4   & 2^3  & 7  & 2^5               & 2^2   & ~  \\ \cline{4-6}
   & ~   &        & 8  & 2^6              & 2^2, 2^3 & ~ \\  
\hline
\end{array}
$$
\end{table}

On the other hand, the single atoms $U$ and $X$ are not legal inputs, since the corresponding output sequences are not sequences of atoms. (The beginning of Table~\ref{TableXUgives}  below shows what happens if the input is the single atom $X$. The output from $X$ has length $16$, and begins $V, W$, but the remaining four output terms do not form an atom.)

\begin{table} [!ht]
\caption{The four input-output combinations needed for the proof. }\label{Table13}
$$
\begin{array}{|c|c|c|}
\hline
\mbox{Input~atoms} & \mbox{Output~atoms} & \mbox{Remarks} \\
\mbox{(number of terms of~} A) & \mbox{(number of terms of~} A) & ~ \\
\hline
T U V W U X U X U & T U V W U X U X U & \mbox{Corresponds~to} \\
                                & V W X U V W X U   &  S R(1) \rightarrow S R(1) R(2) \\
(48)     &  (96) & \\
\hline
V & V W & \mbox{Used~to~prove~that} \\ 
(6) & (12) & R(k) \rightarrow R(k+1)  \\ \cline{1-2}
W & V W & \mbox{for~} k \ge 2  \\ 
(6) & (12) & \\ \cline{1-2}
X U & V W X U & \\
(12) & (24) & \\
\hline
\end{array}
$$
\end{table}

For the proof of the theorem we need just four input sequences, which are listed in Table.~\ref{Table13}.
We  have already verified that the second and third  pair are correct  in Table~\ref{MyFig4b}).

\subsection{The proof of the theorem.}\label{SecBase}

We now give the proof of the theorem.
There are four steps. 

First, Table~\ref{TableSR1} shows that the sequence begins with $S R(1)$ and
that $S R(1) \rightarrow S R(1) R(2)$.

Second, Table~\ref{MyFig4b} shows that $V$ or $W \rightarrow V W$.

Third, Table~\ref{TableXUgives} shows that $X U$ (from $R(k)$ with $k > 1$) $\rightarrow$ $V W X U$ with the ancestor property. 

Fourth, Table~\ref{Table18} shows that 
$$
TUV WUXUXU \rightarrow TUV WUXUXUV WXUV WXU\,,
$$
and combines all the steps to show that $SR(1) \ldots R(k) \rightarrow SR(1) \dots R(k + 1)$.

 This completes the proof.

(Tables~\ref{TableSR1} and \ref{TableXUgives} are given in the Appendix.)


\begin{table}
\caption{Combines all the steps to show that $S R(1) \ldots R(k) \rightarrow  S R(1) \ldots R(k+1)$. } \label{Table18}
\begin{center}
\begin{tabular}{|c|c|c|c|c|}
\hline
Input & Input atoms & Output atoms & Output atoms & Output\\
\cline{1-1}
\cline{2-2}
\cline{3-3}
\cline{4-4}
\cline{5-5}
\multirow{2}{*}{\makecell*{S\\R(1)}} & \multirow{2}{*}{\makecell*{$T U V W U$\\$X U X U$}} & \multirow{2}{*}{\makecell*{$T U V W U X U X U$\\$V W X U V W X U$}} & \makecell*{$T U V W U$\\$X U X U$} & \makecell*{S\\R(1)}\\
\cline{4-4}
\cline{5-5}
 &  &  & \makecell*{$V W X U$\\$V W X U$} & R(2)\\
\cline{1-1}
\cline{2-2}
\cline{3-3}
\cline{4-4}
\cline{5-5}
R(2) & \makecell*{$V W X U$\\$V W X U$} & \makecell*{$(V W)^2 V W X U$\\$(V W)^2 V W X U$} & \makecell*{$(V W)^{2^2 - 1} X U$\\$(V W)^{2^2 - 1} X U$} & R(3)\\
\cline{1-1}
\cline{2-2}
\cline{3-3}
\cline{4-4}
\cline{5-5}
$\cdots$ & $\cdots$ & $\cdots$ & $\cdots$ & $\cdots$\\
\cline{1-1}
\cline{2-2}
\cline{3-3}
\cline{4-4}
\cline{5-5}
R(k) & \makecell*{$(V W)^{2^{k-1}-1} X U$\\$(V W)^{2^{k-1}-1} X U$} & \makecell*{$(V W)^{2^k-2} V W X U$\\$(V W)^{2^k-2} V W X U$} & \makecell*{$(V W)^{2^k-1} X U$\\$(V W)^{2^k-1} X U$} & R(k+1)\\
\hline
\end{tabular}
\end{center}
\end{table}

\begin{coro}\label{Cor1}
The nonzero terms in $A$ have weight $1$ or $2$.
\end{coro}

This is because $A$ is made up of atoms, and the atoms have this property.
It also follows that we never have two weight-$2$ terms in succession, and that weight-$2$ terms only occur at odd indices. Furthermore, except at the start, there are never more than three weight-1 terms in succession.


\section{Numbering the atoms}\label{Sec5}

From the main theorem, we know the order in which the atoms appear in $A$.
We will number them as follows.  

We start by numbering the $U$ atoms
 $U(0), U(1), U(2), \ldots$.
 From $U(2)$ onwards, each $U(k)$ is preceded by an $X$ atom,
 so we number the $X$ atoms $X(2), X(3), X(4), \ldots$. 
 As a result, the block $R(k)$, $k \ge 1$, contains
 $\ldots, X(2k), U(2k) \ldots$, and ends with $\ldots, X(2k+1), U(2k+1)$.

It remains to number  the $V$ and $W$ atoms.
Starting with $R(2)$, the $V$ and $W$ atoms in $R(k)$ are
$V(2k-4,i), W(2k-4,i)$ for $i = 0, \ldots, 2^{k-1}-2$, followed by $X(2k), U(2k)$, then
$V(2k-3,i), W(2k-3,i)$ for $i = 0, \ldots, 2^{k-1}-2$, followed by $X(2k+1), U(2k+1)$.
The initial $V$ and $W$ in the $S$ block we label $V(-1)$ and $W(-1)$.

The resulting list is given in Table~\ref{TableNames}.


\begin{table}[!ht]
\caption{List of the atoms in $A$, in order.}\label{TableNames}
\begin{align}
S ~=~&T, U(0), V(-1), W(-1), U(1), \nonumber \\
R(1) ~=~&X(2), U(2), X(3), U(3), \nonumber \\
R(2) ~=~&V(0,0), W(0,0), X(4), U(4), V(1,0), W(1,0), X(5), U(5),  \nonumber \\
R(3) ~=~&V(2,0), W(2,0), V(2,1), W(2,1), V(2,2), W(2,2), X(6), U(6), \nonumber \\
&V(3,0), W(3,0), V(3,1), W(3,1), V(3,2), W(3,2), X(7), U(7), \nonumber \\
R(4) ~=~&V(4,0), W(4,0), V(4,1), W(4,1), \dots, V(4,6), W(4,6), X(8), U(8), \nonumber \\
&V(5,0), W(5,0), V(5,1), W(5,1), \dots, V(5,6), W(5,6), X(9), U(9), \nonumber \\
R(5) ~=~&V(6,0), W(6,0), V(6,1), W(6,1), \dots, V(6,14), W(6,14), X(10), U(10), \nonumber \\
&V(7,0), W(7,0), V(7,1), W(7,1), \dots, V(7,14), W(7,14), X(11), U(11), \nonumber \\
R(6) ~=~&V(8,0), W(8,0), V(8,1), W(8,1), \dots, V(8,30), W(8,30), X(12), U(12), \nonumber \\
&V(9,0), W(9,0), V(9,1), W(9,1), \dots, V(9,30), W(9,30), X(13), U(13), \nonumber \\
\ldots & \ldots \nonumber
\end{align}
\end{table}

Since we know the lengths of the atoms of each type, it is straightforward to determine where 
each block and each atom begins and ends.
As for the blocks, $S$ begins at $a(1)$ and ends at $a(24)$, and $R(k)$, $k \ge 1$, begins at 
$n = 12 \mu(2k)+1$ and ends at $n = 12 \mu(2k+2)$.

For a given value of $n$, we can find $a(n)$ by first finding which block $n$ is in, and then which atom in that block, and then using the data in the next section.


\section{Explicit formulas for atoms}\label{SecSpectra}
The following theorems specify the individual terms in each type of atom (we call this 
the ``spectrum'' of the atom).

For a given type of atom, the index in $A$ of the first term  is determined by the data in Table~\ref{TableNames},
the positions of the weight 1 terms in the atom are determined by the type of the atom, and the actual
powers of 2 are easy because they start with $2^0$ at the beginning of $A$, and there are no gaps.
It only remains to identify the weight 2 terms, but this can be done recursively using the information in Tables~\ref{TableT}-\ref{TableX}.

In the following sections we also give the term after the atom ends, since that is the first term of the next atom, to help link the atoms together.
 
\textbf{Remark 6.1:}
The atoms of type $V_{k,j}$ and $W_{k,j}$ are very similar, and 
for the purpose of proving the main theorem could have been merged.
However, their spectra are very different. By comparing Theorems~\ref{ThV} and~\ref{ThW}
we see that the spectra of $V_{k,0}$ differs from that of $V_{k,j}$ for $j>0$, whereas for $W_{k,j}$ there is no
such distinction. The chief difference, however, is that the spectrum of $V_{k,j}$ can be expressed simply in terms of $k$ and $j$, while for $W_{k,j}$ one term depends on an earlier atom, which may be of type $V$, $W$, or $X$.

\subsection{Atoms of type U.}\label{SecUSpectra}

 $U(0)$ consists of $a(5)-a(8)$, $U(1)$ is $a(21)-a(24)$, $U(2)$ is $a(33)-a(36)$, $U(3)$ is $a(45)-a(48)$, and so on. $U(0)$ is special, but for $k>0$ there is a uniform formula.
 
\begin{theorem}\label{ThU}
The atom  $U(k)$ for $k > 0$ consists of terms $a(n)$ to $a(n+3)$, where $n = 12 \mu(k+1) - 3$.
Then
$a(n) = 2^r + 2^s$, where $r = 2 \mu(k+1)-1$, and $s = 11$ if $k$ is even or 15 if $k$ is odd;
$a(n+1) = 2^t$, where $t = 8 \mu(k+1) - 1$; $a(n+2) = a(n+1)+1$; and $a(n+3) = a(n+1)+2$.
The next term after $U(k)$ ends is $2^u+2^v$ where $u=2 \mu(k+3)-1$, $v = 2 \mu(k+3)+1$.
\end{theorem}

For example, for $U(3)$ we get $a(45)=2^{15}+2^7 = 32896, a(46) = 2^{31}, a(47) = 2^{32},
a(48) = 2^{33}$, and $a(49) = 2^{16}+2^{17} = 196608$.

 \subsection{Atoms of type X.}\label{SecXSpectra}

 $X(2)$ consists of $a(25)-a(32)$, $X(3)$ is $a(37)-a(44)$, and so on.
 $X(k), k \ge 2$  consists of terms $a(n)-a(n+7)$, where $n = 12 \mu(k+1)  -11$.
 
 The initial $X(k)$ atoms are exceptional, but starting at $X(10)$ there is a uniform formula.
 
 From now on, to save space, we will specify numbers by writing them as powers of 2 and listing the exponents in square brackets. 
 
\begin{theorem}\label{ThX}
The atom  $X(k)$ for $k \ge 10$ consists of terms $a(n)$ to $a(n+7)$, where $n = 12 \mu(k+1) - 11$.
Let $x = \mu(k-1)-1$. 
Then $a(n)= [4x+2, 8x+5], a(n+1) = [16x+10]$, 
$a(n+2)= [16x+11], a(n+3)=[16x+12]$, 
$a(n+4) = [x, 8x+6], a(n+5)  = [16x+14], a(n+6) = [2x+1, 4x+3]$, and $a(n+7) = [16x+15]$.
The next term after $X(k)$ ends is $[11~(k \mbox{~even}) \mbox{~or~} 15~(k \mbox{~odd}), 8x+7]$.
\end{theorem}

For example, for $X(10)$ we get $n=12 \mu(10)-11 = 565$, $x = \mu(9)-1 = 23$, 
 $a(565)=[94,189],  a(566)=[378],  a(567)=[379],
a(568) = [380], a(569) = [23,190],  a(570)= [381],  a(571) = [47,95],
a(572) = [382]$, and $a(573) =  [11,191]$. Note the saving in space: when written in full, $a(565) = [94, 189] = 2^{94}+2^{189}$ is
$$
           784637716923335095479473677920765342641360514956390301696\,.
$$

 \subsection{Atoms of type V.}\label{SecVSpectra}
$V(0,0)$ consists of $a(49)-a(54)$,
$V(1,0)$ is $a(73)-a(78)$,
$V(2,0)$ is $a(97)-a(102)$,
$V(2,1)$ is $a(109)-a(114)$,
$V(2,2)$ is $a(121)-a(126)$,
$V(3,0)$ is $a(145)-a(150)$,
$V(3,1)$ is $a(157)-a(162)$,
$V(3,2)$ is $a(169)-a(174)$,
and so on.

 \begin{theorem}\label{ThV}
The atom  $V(k,j)$ for $k \ge 0, 0 \le j \le 2^{\lfloor k/2 \rfloor +1} -2$ consists of terms $a(n)$ to $a(n+5)$, where $n = 12 \mu(k+4) + 12j + 1$.

If $j=0$ let $x=2 \mu(k+4)$. Then 
$a(n) = [2x, 2x+1], a(n+1) = [4x+2], a(n+2) = [4x+3], a(n+3) = [4x+4],
a(n+4) = [x, 2x+2], a(n+5) = [4x+5]$.
The term after $V(k,0)$ ends is $a(n+6) = [x+1, 2x+3]$.

If $j>0$ let $x= \mu(k+4)+j$. Then 
$a(n) = [2x, 4x+1], a(n+1) = [8x+2], a(n+2) = [8x+3], a(n+3) = [8x+4],
a(n+4) = [x, 4x+2], a(n+5) = [8x+5]$.
The term after $V(k,j)$ ends is $a(n+6) = [2x+1, 4x+3]$.
\end{theorem}

For example, for $V(2,0)$ we get $n=97$, $x=16$, 
$a(97) = [32,33], a(98) = [66], a(99)= [67], a(100) = [68], a(101) = [16,34], a(102)=[69], a(103)=[17,35]$.

For $V(2,1)$ we get $n=109$, $x=9$, 
$a(109) = [18,37], a(110) = [74], a(111)= [75], a(112) = [76], a(113) = [9,38], a(114)=[77], a(115)=[19,39]$.

 \subsection{Atoms of type W.}\label{SecWSpectra}
 
$W(0,0)$ consists of $a(55)-a(60)$,
$W(1,0)$ is $a(79)-a(84)$,
$W(2,0)$ is $a(103)-a(108)$,
$W(2,1)$ is $a(115)-a(120)$,
$W(2,2)$ is $a(127)-a(132)$,
$W(3,0)$ is $a(151)-a(156)$,
$W(3,1)$ is $a(163)-a(168)$,
$W(3,2)$ is $a(175)-a(180)$,
and so on.

$W(k,j)$ begins immediately after $V(k,j)$ ends. 
$W(k,j)$ is the only one of the five atoms whose  terms involve the terms of an earlier atom in a recursive way.
However, in contrast to $V(k,j)$, here there is no need to consider the case $j=0$ separately.

 \begin{theorem}\label{ThW}
The atom  $W(k,j)$ for $k \ge 0, 0 \le j \le 2^{\lfloor k/2 \rfloor +1} -2$ consists of terms $a(n)$ to $a(n+5)$, where $n = 12 \mu(k+4) + 12j + 7$. Let $m = (n+3)/2$. Then $a(m)$ is a weight 2 term in an earlier atom of type $V$, $W$, or $X$ (we give specific information about which earlier atom is involved immediately after the theorem). Let $a(m) = 2^x + 2^y$ where $x > y$. 
Then $x = 2\mu(k+4)+2j+2$,
$a(n) = [x-1,2x-1], a(n+1) = [4x-2], a(n+2) = [4x-1], a(n+3) = [4x],
a(n+4) = [y,2x], a(n+5) = [4x+1]$.
The term after $W(k,j)$ ends  is $a(n+6) = [x, 2x+1]$.
\end{theorem}

The term $a(m)$ mentioned in the theorem belongs to an atom of type $V$, $W$, or $X$.
In each case it is the fifth term in the atom, that is, the term $a(n+4)$  in one
of Theorems~\ref{ThV}, \ref{ThW}, \ref{ThX}, which is always a weight 2 term.
If $j$ has its maximum value, $2^{\lfloor k/2 \rfloor +1} -2$, then $a(m)$ is the fifth term in $X(k+2)$.\footnote{When checking small examples, remember that the formulas for $X(k)$ in Theorem~\ref{ThX} assume that $k \ge 10$.}
Otherwise, for $j \ge 0$, if $j$ is even then $a(m)$ is the fifth term in $V(k-2,\lfloor j/2 \rfloor)$,
and if $j$ is odd, $a(m)$ is the fifth term in $W(k-2,\lfloor j/2 \rfloor)$. Only in this last case do we need to
refer back to an earlier atom of type $W$, which in turn may refer back to a still earlier $W$ atom, $\ldots$

For example, for $W(0,0)$ we get $n=12\mu(4)+7 = 55$, $m=29$,
and   $a(29) = 1025 = 2^{10}+2^0$, so $x=10$, $y=0$.
Then $a(55)=[9,19], a(56)=[38], a(57)=[39], a(58)=[40],
a(59) = [0,20], a(60)=[41]$, and $a(61)=[10,21]$.

\subsection{The graph of the sequence.}\label{SecGraph}
Given the irregular appearance of the initial terms of $A$, the graph of a large number 
of terms is surprisingly regular. A log-linear plot consists
essentially  of two straight lines  with infinitely many isolated points.  (Figure~\ref{FigPlot}
shows a plot of the points $(n, \log_{10}a(n))$ for $1 \le n \le 1500$.)
This is easily explained.  From Table~\ref{TableNames} we see that if $k$ is large,
almost all the atoms are of type $V_{k,j}$ or $W_{k,j}$. 
It follows from Theorems~\ref{ThV} and \ref{ThW}  that in these atoms one-third of the terms
satisfy $\log_2 a(n) \approx n/3$ and two-thirds satisfy $\log_2 a(n) \approx 2n/3$.
This accounts for the majority of points in the graph.
From Theorems~\ref{ThU} and \ref{ThX} we see that most 
terms in the $U_k$ and $X_k$ atoms  also lie on the same
two lines, except that in each $U_k$ and $X_k$ there is one point on the line $\log_2 a(n) \approx n/6$.


\begin{figure}[!ht]
 \centerline{\includegraphics[angle=0, width=6in]{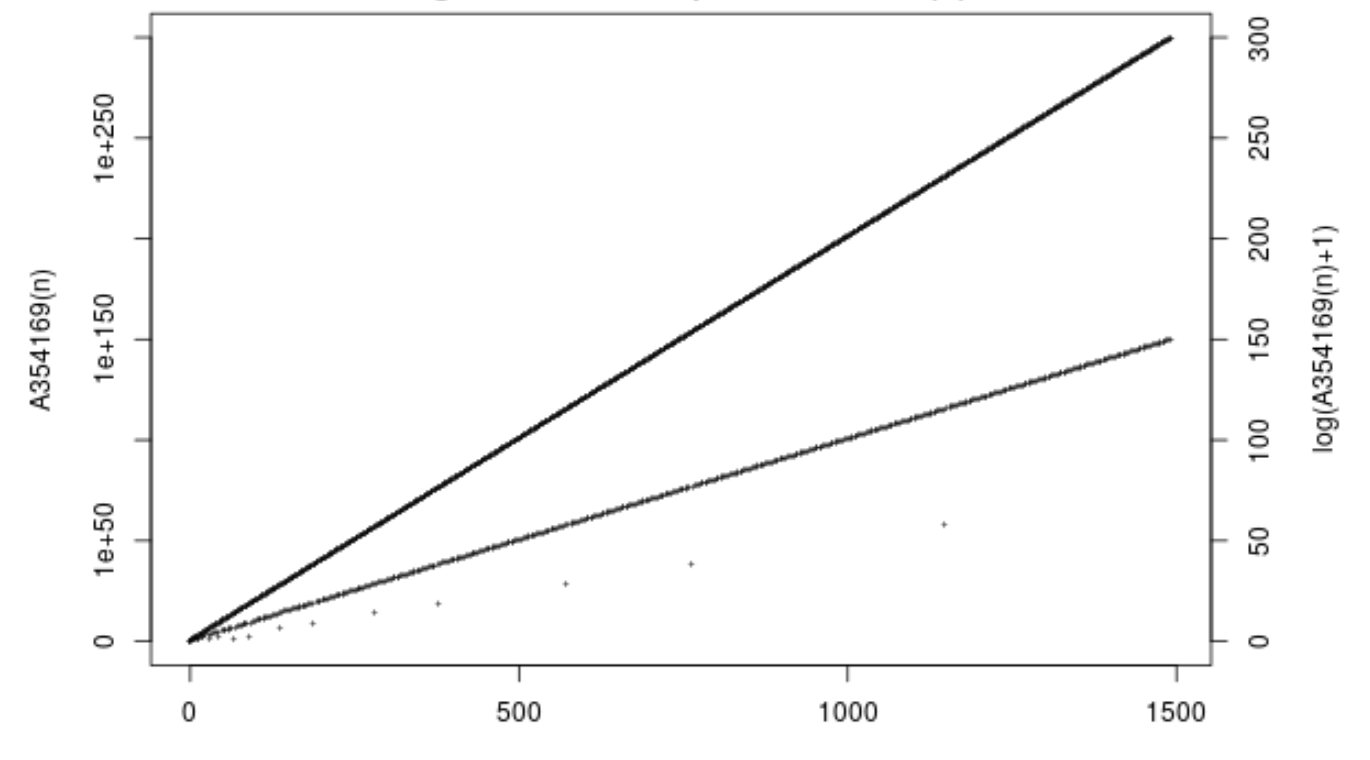}}
 \caption{Log-linear plot of $1500$ terms of $A$. The two lines have slopes
 $2 \log_{10}(2)/3$ and $\log_{10}(2)/3$.}\label{FigPlot}
 \end{figure}


\section{Related sequences}\label{SecSubsid}
Theorem~\ref{MainTh} also enables us to give formulas for various sequences related to $A$.
(These had been conjectured earlier, often in a different and more complicated form.)
 
For example, let $\{\tau(n), n \ge 1\}$ denote the indices of the weight-$2$ terms of $A$.
This sequence (\seqnum{A354798}) begins
$$
5, 9, 13, 15, 19, 21, 25, 29, 31, 33, 37, 41, 43, 45, 49, 53, 55, 59, 61, \ldots.
$$

\begin{theorem}
The generating function 
\beql{EqW2a}
P(x) ~=~ \sum_{n=1}^{\infty} x^{\tau(n)}
\eeq
is equal to 
\beql{EqW2b}
P_S ~+~ \sum_{k=1}^{\infty} ~
x^{12 \cdot 2^k + 1}~
\left(\frac{1-x^{12(2^{k-1}-1)}}{1-x^{12}}~P_{VW} ~+~ 
x^{12 \cdot 2^{k-1} - 1} P_{XU}\right)~\left(1~+~ x^{12 \cdot 2^{k-1}}\right)\,,
\eeq
where $P_S = x^5+x^9+x^{13}+x^{15}+x^{19}+x^{21}$,
$P_{VW} = 1+x^4+x^6+x^{10}$, and
$P_{XU} =1+x^4+x^6+x^8$.

\end{theorem}
\begin{proof}
Since we know from Table~\ref{TableWts} exactly where the weight-$2$ terms are 
in each atom, the proof is just a matter of bookkeeping. 
Generating functions with the sequence in the exponents,\footnote{These are known as  ``lacunary'' power series.}
as in \eqn{EqW2a},
seem to work well for sequences related to \seqnum{A029744}, which makes the
bookkeeping easier.

We compute the generating function separately  for each block of terms $R(k)$, $ k \ge 1$,
and then sum on $k$.

The  $R(k)$ block starts in $A$ at index  $n = 12 \mu(2k)+1$, ends at $n = 12 \mu(2k+2)$,
and contains $2 \mu(2k)$ terms of $A$. Also, $R(k)$ contains $2^{k+1}$ atoms, and $2^{k+2}$ 
terms of weight $2$.

There is a weight-2 term at the start of $R(k)$, which we check is $\tau(2^{k+2}-1) = 12 \mu(2k)+1$.
As we move through a $VW$ block in $R(k)$,  the values of $\tau$ increase by $0, 4, 6, 10$,
which we keep track of by multiplying the generating function  by the polynomial $P_{VW} = 1+x^4+x^6+x^{10}$.
Similarly, as we move through an $XU$ block,  the values of $\tau$ increase by $0, 4, 6, 8$,
corresponding to multiplication by $P_{XU} =1+x^4+x^6+x^8$. 
Every new copy of $VW$ in $R(k)$ requires an additional  multiplication by $x^{12}$.
The remaining parts of \eqn{EqW2b} are easily verified. \end{proof}

For example, the terms for $k=1$ and $k=2$ in the summation in \eqn{EqW2b}  are
$$
x^{25} \left(x^{8}+x^{6}+x^{4}+1\right) \left(x^{12}+1\right)
~=~ x^{25}+x^{29}+x^{31}+x^{33}+x^{37}+x^{41}+x^{43}+x^{45}
$$
and
$$
x^{49} \left(\left(x^{8}+x^{6}+x^{4}+1\right) x^{12}+x^{10}+x^{6}+x^{4}+1\right) \left(x^{24}+1\right) ~=~
$$
$$
x^{49}+x^{53}+x^{55}+x^{59}+x^{61}+x^{65}+x^{67}+x^{69}+x^{73}+x^{77}+x^{79}+x^{83}+x^{85}+x^{89}+x^{91}+x^{93},
$$
and the reader can confirm that these exponents are two successive blocks of terms in \seqnum{A354169}.

As an immediate  corollary we get a similar generating function for the complementary sequence
$\{\sigma(n), n \ge 1\}$  (\seqnum{A354767}) that lists the indices of the weight-$1$ terms of $A$:
\begin{coro}\label{CorWt1}
\beql{EqW1a}
 \sum_{n=1}^{\infty} x^{\sigma(n)} ~=~   x/(1-x) -  P(x)\,.
 \eeq
 \end{coro}
(A more complicated version of this generating function
had previously been proposed in \seqnum{A354767}.)

As a second corollary we obtain an {\em ordinary} generating function for \seqnum{A355150}, the sequence whose $n$th
term $\omega(n)$ gives the Hamming weight of $a(n)$:
\begin{coro}\label{CorWeights}
\beql{EqWeights}
\sum_{n=1}^{\infty} \omega(n) x^n  ~=~  x/(1-x) +  P(x)\,.
\eeq
\end{coro}

Here are two further examples.
Suppose  the $i$-th weight-$2$ term of $A$ is $2^{e(i)} + 2^{f(i)}$, with $e(i) < f(i)$. 
The sequences $\{e(i)\}$ and $\{f(i)\}$  form entries \seqnum{A354773} 
and \seqnum{A354774} respectively. These entries contain conjectured formulas which now follow from our main theorem (we will not repeat them here).



\bigskip
\hrule
\bigskip

Concerned with sequences
\seqnum{A000120},
\seqnum{A029744},
\seqnum{A090252},
\seqnum{A098550},
\seqnum{A121216},
\seqnum{A252867},
\seqnum{A347113},
\seqnum{A353708},
\seqnum{A353712},
\seqnum{A354169},
\seqnum{A354680},
\seqnum{A354767},
\seqnum{A354773},
\seqnum{A354774},
\seqnum{A354798},
\seqnum{A355150}.

\bigskip
\hrule
\bigskip

\noindent 2010 Mathematics Subject Classification 11B37, 11B83


\section{Appendix:  Tables~\ref{TableSR1} and \ref{TableXUgives}.}\label{SecApp}

\begin{longtable}{|l|l|l|l|l|}
\caption{The sequence  $A$ begins with $S R(1)$,
$S R(1) \rightarrow S R(1) R(2)$, and $R(2)$ satisfies the ancestor property.}\label{TableSR1} \\
\hline
Input atoms & Input terms & Free bits & Output terms & Output atoms\\
\cline{1-1}
\cline{2-2}
\cline{3-3}
\cline{4-4}
\cline{5-5}
\multirow{8}{*}{$T_{S}$} & \multirow{2}{*}{$a_{1}=$\colorbox{red!30}{$2^{0}$}} & {\mbox{None}} & $a_{1}=$$2^{0}$ & \multirow{4}{*}{$T_{S}$}\\
\cline{3-3}
\cline{4-4}
 &  & {\mbox{None}} & $a_{2}=$$2^{1}$ & \\
\cline{2-2}
\cline{3-3}
\cline{4-4}
 & \multirow{2}{*}{$a_{2}=$\colorbox{green!30}{$2^{1}$}} & \colorbox{red!30}{$2^{0}$} & $a_{3}=$$2^{2}$ & \\
\cline{3-3}
\cline{4-4}
 &  & \colorbox{red!30}{$2^{0}$} & $a_{4}=$$2^{3}$ & \\
\cline{2-2}
\cline{3-3}
\cline{4-4}
\cline{5-5}
 & \multirow{2}{*}{$a_{3}=$\colorbox{blue!30}{$2^{2}$}} & \colorbox{red!30}{$2^{0}$}, \colorbox{green!30}{$2^{1}$} & $a_{5}=$\colorbox{red!30}{$2^{0}$}$ + $\colorbox{green!30}{$2^{1}$} & \multirow{4}{*}{$U_{S}$}\\
\cline{3-3}
\cline{4-4}
 &  & {\mbox{None}} & $a_{6}=$$2^{4}$ & \\
\cline{2-2}
\cline{3-3}
\cline{4-4}
 & \multirow{2}{*}{$a_{4}=$\colorbox{cyan!30}{$2^{3}$}} & \colorbox{blue!30}{$2^{2}$} & $a_{7}=$$2^{5}$ & \\
\cline{3-3}
\cline{4-4}
 &  & \colorbox{blue!30}{$2^{2}$} & $a_{8}=$$2^{6}$ & \\
\cline{1-1}
\cline{2-2}
\cline{3-3}
\cline{4-4}
\cline{5-5}
\multirow{8}{*}{$U_{S}$} & \multirow{2}{*}{$a_{5}=$\colorbox{red!30}{$2^{0}$}$ + $\colorbox{green!30}{$2^{1}$}} & \colorbox{blue!30}{$2^{2}$}, \colorbox{cyan!30}{$2^{3}$} & $a_{9}=$\colorbox{blue!30}{$2^{2}$}$ + $\colorbox{cyan!30}{$2^{3}$} & \multirow{6}{*}{$V_{S}$}\\
\cline{3-3}
\cline{4-4}
 &  & {\mbox{None}} & $a_{10}=$$2^{7}$ & \\
\cline{2-2}
\cline{3-3}
\cline{4-4}
 & \multirow{2}{*}{$a_{6}=$\colorbox{magenta!30}{$2^{4}$}} & \colorbox{red!30}{$2^{0}$}, \colorbox{green!30}{$2^{1}$} & $a_{11}=$$2^{8}$ & \\
\cline{3-3}
\cline{4-4}
 &  & \colorbox{red!30}{$2^{0}$}, \colorbox{green!30}{$2^{1}$} & $a_{12}=$$2^{9}$ & \\
\cline{2-2}
\cline{3-3}
\cline{4-4}
 & \multirow{2}{*}{$a_{7}=$\colorbox{yellow!30}{$2^{5}$}} & \colorbox{red!30}{$2^{0}$}, \colorbox{green!30}{$2^{1}$}, \colorbox{magenta!30}{$2^{4}$} & $a_{13}=$\colorbox{red!30}{$2^{0}$}$ + $\colorbox{magenta!30}{$2^{4}$} & \\
\cline{3-3}
\cline{4-4}
 &  & \colorbox{green!30}{$2^{1}$} & $a_{14}=$$2^{10}$ & \\
\cline{2-2}
\cline{3-3}
\cline{4-4}
\cline{5-5}
 & \multirow{2}{*}{$a_{8}=$\colorbox{gray!30}{$2^{6}$}} & \colorbox{green!30}{$2^{1}$}, \colorbox{yellow!30}{$2^{5}$} & $a_{15}=$\colorbox{green!30}{$2^{1}$}$ + $\colorbox{yellow!30}{$2^{5}$} & \multirow{6}{*}{$W_{S}$}\\
\cline{3-3}
\cline{4-4}
 &  & {\mbox{None}} & $a_{16}=$$2^{11}$ & \\
\cline{1-1}
\cline{2-2}
\cline{3-3}
\cline{4-4}
\multirow{12}{*}{$V_{S}$} & \multirow{2}{*}{$a_{9}=$\colorbox{blue!30}{$2^{2}$}$ + $\colorbox{cyan!30}{$2^{3}$}} & \colorbox{gray!30}{$2^{6}$} & $a_{17}=$$2^{12}$ & \\
\cline{3-3}
\cline{4-4}
 &  & \colorbox{gray!30}{$2^{6}$} & $a_{18}=$$2^{13}$ & \\
\cline{2-2}
\cline{3-3}
\cline{4-4}
 & \multirow{2}{*}{$a_{10}=$\colorbox{brown!30}{$2^{7}$}} & \colorbox{blue!30}{$2^{2}$}, \colorbox{cyan!30}{$2^{3}$}, \colorbox{gray!30}{$2^{6}$} & $a_{19}=$\colorbox{blue!30}{$2^{2}$}$ + $\colorbox{gray!30}{$2^{6}$} & \\
\cline{3-3}
\cline{4-4}
 &  & \colorbox{cyan!30}{$2^{3}$} & $a_{20}=$$2^{14}$ & \\
\cline{2-2}
\cline{3-3}
\cline{4-4}
\cline{5-5}
 & \multirow{2}{*}{$a_{11}=$\colorbox{lime!30}{$2^{8}$}} & \colorbox{cyan!30}{$2^{3}$}, \colorbox{brown!30}{$2^{7}$} & $a_{21}=$\colorbox{cyan!30}{$2^{3}$}$ + $\colorbox{brown!30}{$2^{7}$} & \multirow{4}{*}{$U_{S}$}\\
\cline{3-3}
\cline{4-4}
 &  & {\mbox{None}} & $a_{22}=$$2^{15}$ & \\
\cline{2-2}
\cline{3-3}
\cline{4-4}
 & \multirow{2}{*}{$a_{12}=$\colorbox{olive!30}{$2^{9}$}} & \colorbox{lime!30}{$2^{8}$} & $a_{23}=$$2^{16}$ & \\
\cline{3-3}
\cline{4-4}
 &  & \colorbox{lime!30}{$2^{8}$} & $a_{24}=$$2^{17}$ & \\
\cline{2-2}
\cline{3-3}
\cline{4-4}
\cline{5-5}
 & \multirow{2}{*}{$a_{13}=$\colorbox{red!30}{$2^{0}$}$ + $\colorbox{magenta!30}{$2^{4}$}} & \colorbox{lime!30}{$2^{8}$}, \colorbox{olive!30}{$2^{9}$} & $a_{25}=$\colorbox{lime!30}{$2^{8}$}$ + $\colorbox{olive!30}{$2^{9}$} & \multirow{8}{*}{$X_{R(1)}$}\\
\cline{3-3}
\cline{4-4}
 &  & {\mbox{None}} & $a_{26}=$$2^{18}$ & \\
\cline{2-2}
\cline{3-3}
\cline{4-4}
 & \multirow{2}{*}{$a_{14}=$\colorbox{orange!30}{$2^{10}$}} & \colorbox{red!30}{$2^{0}$}, \colorbox{magenta!30}{$2^{4}$} & $a_{27}=$$2^{19}$ & \\
\cline{3-3}
\cline{4-4}
 &  & \colorbox{red!30}{$2^{0}$}, \colorbox{magenta!30}{$2^{4}$} & $a_{28}=$$2^{20}$ & \\
\cline{1-1}
\cline{2-2}
\cline{3-3}
\cline{4-4}
\multirow{12}{*}{$W_{S}$} & \multirow{2}{*}{$a_{15}=$\colorbox{green!30}{$2^{1}$}$ + $\colorbox{yellow!30}{$2^{5}$}} & \colorbox{red!30}{$2^{0}$}, \colorbox{magenta!30}{$2^{4}$}, \colorbox{orange!30}{$2^{10}$} & $a_{29}=$\colorbox{red!30}{$2^{0}$}$ + $\colorbox{orange!30}{$2^{10}$} & \\
\cline{3-3}
\cline{4-4}
 &  & \colorbox{magenta!30}{$2^{4}$} & $a_{30}=$$2^{21}$ & \\
\cline{2-2}
\cline{3-3}
\cline{4-4}
 & \multirow{2}{*}{$a_{16}=$\colorbox{pink!30}{$2^{11}$}} & \colorbox{green!30}{$2^{1}$}, \colorbox{magenta!30}{$2^{4}$}, \colorbox{yellow!30}{$2^{5}$} & $a_{31}=$\colorbox{green!30}{$2^{1}$}$ + $\colorbox{magenta!30}{$2^{4}$} & \\
\cline{3-3}
\cline{4-4}
 &  & \colorbox{yellow!30}{$2^{5}$} & $a_{32}=$$2^{22}$ & \\
\cline{2-2}
\cline{3-3}
\cline{4-4}
\cline{5-5}
 & \multirow{2}{*}{$a_{17}=$\colorbox{purple!30}{$2^{12}$}} & \colorbox{yellow!30}{$2^{5}$}, \colorbox{pink!30}{$2^{11}$} & $a_{33}=$\colorbox{yellow!30}{$2^{5}$}$ + $\colorbox{pink!30}{$2^{11}$} & \multirow{4}{*}{$U_{R(1)}$}\\
\cline{3-3}
\cline{4-4}
 &  & {\mbox{None}} & $a_{34}=$$2^{23}$ & \\
\cline{2-2}
\cline{3-3}
\cline{4-4}
 & \multirow{2}{*}{$a_{18}=$\colorbox{teal!30}{$2^{13}$}} & \colorbox{purple!30}{$2^{12}$} & $a_{35}=$$2^{24}$ & \\
\cline{3-3}
\cline{4-4}
 &  & \colorbox{purple!30}{$2^{12}$} & $a_{36}=$$2^{25}$ & \\
\cline{2-2}
\cline{3-3}
\cline{4-4}
\cline{5-5}
 & \multirow{2}{*}{$a_{19}=$\colorbox{blue!30}{$2^{2}$}$ + $\colorbox{gray!30}{$2^{6}$}} & \colorbox{purple!30}{$2^{12}$}, \colorbox{teal!30}{$2^{13}$} & $a_{37}=$\colorbox{purple!30}{$2^{12}$}$ + $\colorbox{teal!30}{$2^{13}$} & \multirow{8}{*}{$X_{R(1)}$}\\
\cline{3-3}
\cline{4-4}
 &  & {\mbox{None}} & $a_{38}=$$2^{26}$ & \\
\cline{2-2}
\cline{3-3}
\cline{4-4}
 & \multirow{2}{*}{$a_{20}=$\colorbox{violet!30}{$2^{14}$}} & \colorbox{blue!30}{$2^{2}$}, \colorbox{gray!30}{$2^{6}$} & $a_{39}=$$2^{27}$ & \\
\cline{3-3}
\cline{4-4}
 &  & \colorbox{blue!30}{$2^{2}$}, \colorbox{gray!30}{$2^{6}$} & $a_{40}=$$2^{28}$ & \\
\cline{1-1}
\cline{2-2}
\cline{3-3}
\cline{4-4}
\multirow{8}{*}{$U_{S}$} & \multirow{2}{*}{$a_{21}=$\colorbox{cyan!30}{$2^{3}$}$ + $\colorbox{brown!30}{$2^{7}$}} & \colorbox{blue!30}{$2^{2}$}, \colorbox{gray!30}{$2^{6}$}, \colorbox{violet!30}{$2^{14}$} & $a_{41}=$\colorbox{blue!30}{$2^{2}$}$ + $\colorbox{violet!30}{$2^{14}$} & \\
\cline{3-3}
\cline{4-4}
 &  & \colorbox{gray!30}{$2^{6}$} & $a_{42}=$$2^{29}$ & \\
\cline{2-2}
\cline{3-3}
\cline{4-4}
 & \multirow{2}{*}{$a_{22}=$\colorbox{red!30}{$2^{15}$}} & \colorbox{cyan!30}{$2^{3}$}, \colorbox{gray!30}{$2^{6}$}, \colorbox{brown!30}{$2^{7}$} & $a_{43}=$\colorbox{cyan!30}{$2^{3}$}$ + $\colorbox{gray!30}{$2^{6}$} & \\
\cline{3-3}
\cline{4-4}
 &  & \colorbox{brown!30}{$2^{7}$} & $a_{44}=$$2^{30}$ & \\
\cline{2-2}
\cline{3-3}
\cline{4-4}
\cline{5-5}
 & \multirow{2}{*}{$a_{23}=$\colorbox{green!30}{$2^{16}$}} & \colorbox{brown!30}{$2^{7}$}, \colorbox{red!30}{$2^{15}$} & $a_{45}=$\colorbox{brown!30}{$2^{7}$}$ + $\colorbox{red!30}{$2^{15}$} & \multirow{4}{*}{$U_{R(1)}$}\\
\cline{3-3}
\cline{4-4}
 &  & {\mbox{None}} & $a_{46}=$$2^{31}$ & \\
\cline{2-2}
\cline{3-3}
\cline{4-4}
 & \multirow{2}{*}{$a_{24}=$\colorbox{blue!30}{$2^{17}$}} & \colorbox{green!30}{$2^{16}$} & $a_{47}=$$2^{32}$ & \\
\cline{3-3}
\cline{4-4}
 &  & \colorbox{green!30}{$2^{16}$} & $a_{48}=$$2^{33}$ & \\
\cline{1-1}
\cline{2-2}
\cline{3-3}
\cline{4-4}
\cline{5-5}
\multirow{16}{*}{$X_{R(1)}$} & \multirow{2}{*}{$a_{25}=$\colorbox{lime!30}{$2^{8}$}$ + $\colorbox{olive!30}{$2^{9}$}} & \colorbox{green!30}{$2^{16}$}, \colorbox{blue!30}{$2^{17}$} & $a_{49}=$\colorbox{green!30}{$2^{16}$}$ + $\colorbox{blue!30}{$2^{17}$} & \multirow{6}{*}{$V_{R(2)}$}\\
\cline{3-3}
\cline{4-4}
 &  & {\mbox{None}} & $a_{50}=$$2^{34}$ & \\
\cline{2-2}
\cline{3-3}
\cline{4-4}
 & \multirow{2}{*}{$a_{26}=$\colorbox{cyan!30}{$2^{18}$}} & \colorbox{lime!30}{$2^{8}$}, \colorbox{olive!30}{$2^{9}$} & $a_{51}=$$2^{35}$ & \\
\cline{3-3}
\cline{4-4}
 &  & \colorbox{lime!30}{$2^{8}$}, \colorbox{olive!30}{$2^{9}$} & $a_{52}=$$2^{36}$ & \\
\cline{2-2}
\cline{3-3}
\cline{4-4}
 & \multirow{2}{*}{$a_{27}=$\colorbox{magenta!30}{$2^{19}$}} & \colorbox{lime!30}{$2^{8}$}, \colorbox{olive!30}{$2^{9}$}, \colorbox{cyan!30}{$2^{18}$} & $a_{53}=$\colorbox{lime!30}{$2^{8}$}$ + $\colorbox{cyan!30}{$2^{18}$} & \\
\cline{3-3}
\cline{4-4}
 &  & \colorbox{olive!30}{$2^{9}$} & $a_{54}=$$2^{37}$ & \\
\cline{2-2}
\cline{3-3}
\cline{4-4}
\cline{5-5}
 & \multirow{2}{*}{$a_{28}=$\colorbox{yellow!30}{$2^{20}$}} & \colorbox{olive!30}{$2^{9}$}, \colorbox{magenta!30}{$2^{19}$} & $a_{55}=$\colorbox{olive!30}{$2^{9}$}$ + $\colorbox{magenta!30}{$2^{19}$} & \multirow{6}{*}{$W_{R(2)}$}\\
\cline{3-3}
\cline{4-4}
 &  & {\mbox{None}} & $a_{56}=$$2^{38}$ & \\
\cline{2-2}
\cline{3-3}
\cline{4-4}
 & \multirow{2}{*}{$a_{29}=$\colorbox{red!30}{$2^{0}$}$ + $\colorbox{orange!30}{$2^{10}$}} & \colorbox{yellow!30}{$2^{20}$} & $a_{57}=$$2^{39}$ & \\
\cline{3-3}
\cline{4-4}
 &  & \colorbox{yellow!30}{$2^{20}$} & $a_{58}=$$2^{40}$ & \\
\cline{2-2}
\cline{3-3}
\cline{4-4}
 & \multirow{2}{*}{$a_{30}=$\colorbox{gray!30}{$2^{21}$}} & \colorbox{red!30}{$2^{0}$}, \colorbox{orange!30}{$2^{10}$}, \colorbox{yellow!30}{$2^{20}$} & $a_{59}=$\colorbox{red!30}{$2^{0}$}$ + $\colorbox{yellow!30}{$2^{20}$} & \\
\cline{3-3}
\cline{4-4}
 &  & \colorbox{orange!30}{$2^{10}$} & $a_{60}=$$2^{41}$ & \\
\cline{2-2}
\cline{3-3}
\cline{4-4}
\cline{5-5}
 & \multirow{2}{*}{$a_{31}=$\colorbox{green!30}{$2^{1}$}$ + $\colorbox{magenta!30}{$2^{4}$}} & \colorbox{orange!30}{$2^{10}$}, \colorbox{gray!30}{$2^{21}$} & $a_{61}=$\colorbox{orange!30}{$2^{10}$}$ + $\colorbox{gray!30}{$2^{21}$} & \multirow{8}{*}{$X_{R(2)}$}\\
\cline{3-3}
\cline{4-4}
 &  & {\mbox{None}} & $a_{62}=$$2^{42}$ & \\
\cline{2-2}
\cline{3-3}
\cline{4-4}
 & \multirow{2}{*}{$a_{32}=$\colorbox{brown!30}{$2^{22}$}} & \colorbox{green!30}{$2^{1}$}, \colorbox{magenta!30}{$2^{4}$} & $a_{63}=$$2^{43}$ & \\
\cline{3-3}
\cline{4-4}
 &  & \colorbox{green!30}{$2^{1}$}, \colorbox{magenta!30}{$2^{4}$} & $a_{64}=$$2^{44}$ & \\
\cline{1-1}
\cline{2-2}
\cline{3-3}
\cline{4-4}
\multirow{8}{*}{$U_{R(1)}$} & \multirow{2}{*}{$a_{33}=$\colorbox{yellow!30}{$2^{5}$}$ + $\colorbox{pink!30}{$2^{11}$}} & \colorbox{green!30}{$2^{1}$}, \colorbox{magenta!30}{$2^{4}$}, \colorbox{brown!30}{$2^{22}$} & $a_{65}=$\colorbox{green!30}{$2^{1}$}$ + $\colorbox{brown!30}{$2^{22}$} & \\
\cline{3-3}
\cline{4-4}
 &  & \colorbox{magenta!30}{$2^{4}$} & $a_{66}=$$2^{45}$ & \\
\cline{2-2}
\cline{3-3}
\cline{4-4}
 & \multirow{2}{*}{$a_{34}=$\colorbox{lime!30}{$2^{23}$}} & \colorbox{magenta!30}{$2^{4}$}, \colorbox{yellow!30}{$2^{5}$}, \colorbox{pink!30}{$2^{11}$} & $a_{67}=$\colorbox{magenta!30}{$2^{4}$}$ + $\colorbox{yellow!30}{$2^{5}$} & \\
\cline{3-3}
\cline{4-4}
 &  & \colorbox{pink!30}{$2^{11}$} & $a_{68}=$$2^{46}$ & \\
\cline{2-2}
\cline{3-3}
\cline{4-4}
\cline{5-5}
 & \multirow{2}{*}{$a_{35}=$\colorbox{olive!30}{$2^{24}$}} & \colorbox{pink!30}{$2^{11}$}, \colorbox{lime!30}{$2^{23}$} & $a_{69}=$\colorbox{pink!30}{$2^{11}$}$ + $\colorbox{lime!30}{$2^{23}$} & \multirow{4}{*}{$U_{R(2)}$ [See Note 1]}\\
\cline{3-3}
\cline{4-4}
 &  & {\mbox{None}} & $a_{70}=$$2^{47}$ & \\
\cline{2-2}
\cline{3-3}
\cline{4-4}
 & \multirow{2}{*}{$a_{36}=$\colorbox{orange!30}{$2^{25}$}} & \colorbox{olive!30}{$2^{24}$} & $a_{71}=$$2^{48}$ & \\
\cline{3-3}
\cline{4-4}
 &  & \colorbox{olive!30}{$2^{24}$} & $a_{72}=$$2^{49}$ & \\
\cline{1-1}
\cline{2-2}
\cline{3-3}
\cline{4-4}
\cline{5-5}
\multirow{16}{*}{$X_{R(1)}$} & \multirow{2}{*}{$a_{37}=$\colorbox{purple!30}{$2^{12}$}$ + $\colorbox{teal!30}{$2^{13}$}} & \colorbox{olive!30}{$2^{24}$}, \colorbox{orange!30}{$2^{25}$} & $a_{73}=$\colorbox{olive!30}{$2^{24}$}$ + $\colorbox{orange!30}{$2^{25}$} & \multirow{6}{*}{$V_{R(2)}$}\\
\cline{3-3}
\cline{4-4}
 &  & {\mbox{None}} & $a_{74}=$$2^{50}$ & \\
\cline{2-2}
\cline{3-3}
\cline{4-4}
 & \multirow{2}{*}{$a_{38}=$\colorbox{pink!30}{$2^{26}$}} & \colorbox{purple!30}{$2^{12}$}, \colorbox{teal!30}{$2^{13}$} & $a_{75}=$$2^{51}$ & \\
\cline{3-3}
\cline{4-4}
 &  & \colorbox{purple!30}{$2^{12}$}, \colorbox{teal!30}{$2^{13}$} & $a_{76}=$$2^{52}$ & \\
\cline{2-2}
\cline{3-3}
\cline{4-4}
 & \multirow{2}{*}{$a_{39}=$\colorbox{purple!30}{$2^{27}$}} & \colorbox{purple!30}{$2^{12}$}, \colorbox{teal!30}{$2^{13}$}, \colorbox{pink!30}{$2^{26}$} & $a_{77}=$\colorbox{purple!30}{$2^{12}$}$ + $\colorbox{pink!30}{$2^{26}$} & \\
\cline{3-3}
\cline{4-4}
 &  & \colorbox{teal!30}{$2^{13}$} & $a_{78}=$$2^{53}$ & \\
\cline{2-2}
\cline{3-3}
\cline{4-4}
\cline{5-5}
 & \multirow{2}{*}{$a_{40}=$\colorbox{teal!30}{$2^{28}$}} & \colorbox{teal!30}{$2^{13}$}, \colorbox{purple!30}{$2^{27}$} & $a_{79}=$\colorbox{teal!30}{$2^{13}$}$ + $\colorbox{purple!30}{$2^{27}$} & \multirow{6}{*}{$W_{R(2)}$}\\
\cline{3-3}
\cline{4-4}
 &  & {\mbox{None}} & $a_{80}=$$2^{54}$ & \\
\cline{2-2}
\cline{3-3}
\cline{4-4}
 & \multirow{2}{*}{$a_{41}=$\colorbox{blue!30}{$2^{2}$}$ + $\colorbox{violet!30}{$2^{14}$}} & \colorbox{teal!30}{$2^{28}$} & $a_{81}=$$2^{55}$ & \\
\cline{3-3}
\cline{4-4}
 &  & \colorbox{teal!30}{$2^{28}$} & $a_{82}=$$2^{56}$ & \\
\cline{2-2}
\cline{3-3}
\cline{4-4}
 & \multirow{2}{*}{$a_{42}=$\colorbox{violet!30}{$2^{29}$}} & \colorbox{blue!30}{$2^{2}$}, \colorbox{violet!30}{$2^{14}$}, \colorbox{teal!30}{$2^{28}$} & $a_{83}=$\colorbox{blue!30}{$2^{2}$}$ + $\colorbox{teal!30}{$2^{28}$} & \\
\cline{3-3}
\cline{4-4}
 &  & \colorbox{violet!30}{$2^{14}$} & $a_{84}=$$2^{57}$ & \\
\cline{2-2}
\cline{3-3}
\cline{4-4}
\cline{5-5}
 & \multirow{2}{*}{$a_{43}=$\colorbox{cyan!30}{$2^{3}$}$ + $\colorbox{gray!30}{$2^{6}$}} & \colorbox{violet!30}{$2^{14}$}, \colorbox{violet!30}{$2^{29}$} & $a_{85}=$\colorbox{violet!30}{$2^{14}$}$ + $\colorbox{violet!30}{$2^{29}$} & \multirow{8}{*}{$X_{R(2)}$}\\
\cline{3-3}
\cline{4-4}
 &  & {\mbox{None}} & $a_{86}=$$2^{58}$ & \\
\cline{2-2}
\cline{3-3}
\cline{4-4}
 & \multirow{2}{*}{$a_{44}=$\colorbox{red!30}{$2^{30}$}} & \colorbox{cyan!30}{$2^{3}$}, \colorbox{gray!30}{$2^{6}$} & $a_{87}=$$2^{59}$ & \\
\cline{3-3}
\cline{4-4}
 &  & \colorbox{cyan!30}{$2^{3}$}, \colorbox{gray!30}{$2^{6}$} & $a_{88}=$$2^{60}$ & \\
\cline{1-1}
\cline{2-2}
\cline{3-3}
\cline{4-4}
\multirow{8}{*}{$U_{R(1)}$} & \multirow{2}{*}{$a_{45}=$\colorbox{brown!30}{$2^{7}$}$ + $\colorbox{red!30}{$2^{15}$}} & \colorbox{cyan!30}{$2^{3}$}, \colorbox{gray!30}{$2^{6}$}, \colorbox{red!30}{$2^{30}$} & $a_{89}=$\colorbox{cyan!30}{$2^{3}$}$ + $\colorbox{red!30}{$2^{30}$} & \\
\cline{3-3}
\cline{4-4}
 &  & \colorbox{gray!30}{$2^{6}$} & $a_{90}=$$2^{61}$ & \\
\cline{2-2}
\cline{3-3}
\cline{4-4}
 & \multirow{2}{*}{$a_{46}=$\colorbox{green!30}{$2^{31}$}} & \colorbox{gray!30}{$2^{6}$}, \colorbox{brown!30}{$2^{7}$}, \colorbox{red!30}{$2^{15}$} & $a_{91}=$\colorbox{gray!30}{$2^{6}$}$ + $\colorbox{brown!30}{$2^{7}$} & \\
\cline{3-3}
\cline{4-4}
 &  & \colorbox{red!30}{$2^{15}$} & $a_{92}=$$2^{62}$ & \\
\cline{2-2}
\cline{3-3}
\cline{4-4}
\cline{5-5}
 & \multirow{2}{*}{$a_{47}=$\colorbox{blue!30}{$2^{32}$}} & \colorbox{red!30}{$2^{15}$}, \colorbox{green!30}{$2^{31}$} & $a_{93}=$\colorbox{red!30}{$2^{15}$}$ + $\colorbox{green!30}{$2^{31}$} & \multirow{4}{*}{$U_{R(2)}$ [See Note 2]}\\
\cline{3-3}
\cline{4-4}
 &  & {\mbox{None}} & $a_{94}=$$2^{63}$ & \\
\cline{2-2}
\cline{3-3}
\cline{4-4}
 & \multirow{2}{*}{$a_{48}=$\colorbox{cyan!30}{$2^{33}$}} & \colorbox{blue!30}{$2^{32}$} & $a_{95}=$$2^{64}$ & \\
\cline{3-3}
\cline{4-4}
 &  & \colorbox{blue!30}{$2^{32}$} & $a_{96}=$$2^{65}$ & \\
\cline{1-1}
\cline{2-2}
\cline{3-3}
\cline{4-4}
\cline{5-5}
 &  & \colorbox{blue!30}{$2^{32}$}, \colorbox{cyan!30}{$2^{33}$} &  & \\
\hline
\end{longtable}

Note 1: We see that the ancestor property is satisfied for the terms $a_{67}$ through $a_{72}$:
	$4 < 5$, $11 < 23$, $5 < 23$;
	$a_{33} = 2^5 + 2^{11}$ is the common ancestor;
	$2^{23}$ has only been combined with $2^{11}$ so far, so we can add $2^{23}$ and $2^{5}$ to form a new term.

\noindent
Note 2: We see that the ancestor property is satisfied for the terms $a_{91}$ through $a_{96}$:
	$6 < 7$, $15 < 31$, $7 < 31$;
	$a_{45} = 2^7 + 2^{15}$ is the common ancestor;
	$2^{31}$ has only been combined with $2^{15}$ so far, so we can add $2^{31}$ and $2^{7}$ to form a new term.


\begin{longtable}[h]{|p{2cm}|p{3cm}|p{3cm}|p{3cm}|p{2cm}|}
\caption{ $X U$ (from $R(k)$ with $k > 1$) $\rightarrow$ $V W X U$ with the ancestor property. }\label{TableXUgives} \\
\hline
Input atoms & Input terms & Free bits & Output terms & Output atoms\\
\cline{1-1}
\cline{2-2}
\cline{3-3}
\cline{4-4}
\cline{5-5}
\multirow{16}{*}{X} & \multirow{2}{*}{\colorbox{green!30}{$2^{b}$} + \colorbox{blue!30}{$2^{c}$} \newline ($b < c$)} & \colorbox{red!30}{$2^{o}$}, \colorbox{green!30}{$2^{p}$} \newline (From the prior atom) & \colorbox{red!30}{$2^{o}$} + \colorbox{green!30}{$2^{p}$} & \multirow{6}{*}{V}\\
\cline{3-3}
\cline{4-4}
 &  & \mbox{None} & $2^{r}$ & \\
\cline{2-2}
\cline{3-3}
\cline{4-4}
 & \multirow{2}{*}{\colorbox{olive!30}{$2^{j}$}} & \colorbox{green!30}{$2^{b}$}, \colorbox{blue!30}{$2^{c}$} & $2^{r+1}$ & \\
\cline{3-3}
\cline{4-4}
 &  & \colorbox{green!30}{$2^{b}$}, \colorbox{blue!30}{$2^{c}$} & $2^{r+2}$ & \\
\cline{2-2}
\cline{3-3}
\cline{4-4}
 & \multirow{2}{*}{\colorbox{orange!30}{$2^{j+1}$}} & \colorbox{green!30}{$2^{b}$}, \colorbox{blue!30}{$2^{c}$}, \colorbox{olive!30}{$2^{j}$} \newline (We can combine $2^j$ with anything; we combine $2^j$ with the least of $2^b$, $2^c$) & \colorbox{green!30}{$2^{b}$} + \colorbox{olive!30}{$2^{j}$} & \\
\cline{3-3}
\cline{4-4}
 &  & \colorbox{blue!30}{$2^{c}$} & $2^{r+3}$ & \\
\cline{2-2}
\cline{3-3}
\cline{4-4}
\cline{5-5}
 & \multirow{2}{*}{\colorbox{pink!30}{$2^{j+2}$}} & \colorbox{blue!30}{$2^{c}$}, \colorbox{orange!30}{$2^{j+1}$} \newline (We can combine $2^{j+1}$ with anything) & \colorbox{blue!30}{$2^{c}$} + \colorbox{orange!30}{$2^{j+1}$} & \multirow{6}{*}{W}\\
\cline{3-3}
\cline{4-4}
 &  & \mbox{None} & $2^{r+4}$ & \\
\cline{2-2}
\cline{3-3}
\cline{4-4}
 & \multirow{2}{*}{\colorbox{cyan!30}{$2^{d}$} + \colorbox{magenta!30}{$2^{e}$} \newline ($d < e$)} & \colorbox{pink!30}{$2^{j+2}$} & $2^{r+5}$ & \\
\cline{3-3}
\cline{4-4}
 &  & \colorbox{pink!30}{$2^{j+2}$} & $2^{r+6}$ & \\
\cline{2-2}
\cline{3-3}
\cline{4-4}
 & \multirow{2}{*}{\colorbox{purple!30}{$2^{j+3}$}} & \colorbox{cyan!30}{$2^{d}$}, \colorbox{magenta!30}{$2^{e}$}, \colorbox{pink!30}{$2^{j+2}$} \newline (We can combine $2^{j+2}$ with anything; and $2^{j+2}$ with the least of $2^d$, $2^e$) & \colorbox{cyan!30}{$2^{d}$} + \colorbox{pink!30}{$2^{j+2}$} & \\
\cline{3-3}
\cline{4-4}
 &  & \colorbox{magenta!30}{$2^{e}$} & $2^{r+7}$ & \\
\cline{2-2}
\cline{3-3}
\cline{4-4}
\cline{5-5}
 & \multirow{2}{*}{\colorbox{yellow!30}{$2^{f}$} + \colorbox{gray!30}{$2^{g}$} \newline ($f < g$)} & \colorbox{magenta!30}{$2^{e}$}, \colorbox{purple!30}{$2^{j+3}$} \newline (We can combine $2^{j+3}$ with anything) & \colorbox{magenta!30}{$2^{e}$} + \colorbox{purple!30}{$2^{j+3}$} & \multirow{8}{*}{X}\\
\cline{3-3}
\cline{4-4}
 &  & \mbox{None} & $2^{r+8}$ & \\
\cline{2-2}
\cline{3-3}
\cline{4-4}
 & \multirow{2}{*}{\colorbox{teal!30}{$2^{j+4}$}} & \colorbox{yellow!30}{$2^{f}$}, \colorbox{gray!30}{$2^{g}$} & $2^{r+9}$ & \\
\cline{3-3}
\cline{4-4}
 &  & \colorbox{yellow!30}{$2^{f}$}, \colorbox{gray!30}{$2^{g}$} & $2^{r+10}$ & \\
\cline{1-1}
\cline{2-2}
\cline{3-3}
\cline{4-4}
\multirow{8}{*}{U} & \multirow{2}{*}{\colorbox{brown!30}{$2^{h}$} + \colorbox{lime!30}{$2^{i}$} \newline ($g, h < i$)} & \colorbox{yellow!30}{$2^{f}$}, \colorbox{gray!30}{$2^{g}$}, \colorbox{teal!30}{$2^{j+4}$} \newline (We can combine $2^{j+4}$ with anything; and $2^{j+4}$ with the least of $2^f$, $2^g$) & \colorbox{yellow!30}{$2^{f}$} + \colorbox{teal!30}{$2^{j+4}$} & \\
\cline{3-3}
\cline{4-4}
 &  & \colorbox{gray!30}{$2^{g}$} & $2^{r+11}$ & \\
\cline{2-2}
\cline{3-3}
\cline{4-4}
 & \multirow{2}{*}{\colorbox{violet!30}{$2^{j+5}$}} & \colorbox{gray!30}{$2^{g}$}, \colorbox{brown!30}{$2^{h}$}, \colorbox{lime!30}{$2^{i}$} \newline (By the ancestor property: $2^h + 2^g$ is the ancestor term, and $2^g + 2^i$ is possible) & \colorbox{gray!30}{$2^{g}$} + \colorbox{lime!30}{$2^{i}$} \newline (Note that $g < i$) & \\
\cline{3-3}
\cline{4-4}
 &  & \colorbox{brown!30}{$2^{h}$} & $2^{r+12}$ & \\
\cline{2-2}
\cline{3-3}
\cline{4-4}
\cline{5-5}
 & \multirow{2}{*}{\colorbox{lightgray!30}{$2^{j+6}$}} & \colorbox{brown!30}{$2^{h}$}, \colorbox{violet!30}{$2^{j+5}$} \newline (We can combine $2^{j+5}$ with anything) & \colorbox{brown!30}{$2^{h}$} + \colorbox{violet!30}{$2^{j+5}$} \newline (Note that $h, i < j+5$, and that $2^h + 2^i$ is the new ancestor term, also $2^{j+5}$ has so far only been used with $2^h$,
 so we can combine it later with $2^i$ to form a new term, 
 and the ancestor property is satisfied in the output atoms) & \multirow{4}{*}{U}\\
\cline{3-3}
\cline{4-4}
 &  & \mbox{None} & $2^{r+13}$ & \\
\cline{2-2}
\cline{3-3}
\cline{4-4}
 & \multirow{2}{*}{\colorbox{darkgray!30}{$2^{j+7}$}} & \colorbox{lightgray!30}{$2^{j+6}$} & $2^{r+14}$ & \\
\cline{3-3}
\cline{4-4}
 &  & \colorbox{lightgray!30}{$2^{j+6}$} & $2^{r+15}$ & \\
\cline{1-1}
\cline{2-2}
\cline{3-3}
\cline{4-4}
\cline{5-5}
 &  & \colorbox{lightgray!30}{$2^{j+6}$}, \colorbox{darkgray!30}{$2^{j+7}$} \newline (For the next atom) &  & \\
\hline
\end{longtable}

\end{document}